\documentclass[a4paper,11pt]{amsart}
\usepackage{a4wide}
\usepackage{amssymb}
\usepackage{xcolor}
\usepackage{subfig}
\usepackage[hidelinks]{hyperref}
\usepackage{graphicx}
\usepackage{overpic}
\usepackage{breakcites}
\allowdisplaybreaks
\usepackage{stmaryrd}
\SetSymbolFont{stmry}{bold}{U}{stmry}{m}{n}
\usepackage{bm}

\newcommand{\Z}{\mathbb{Z}}
\newcommand{\R}{\mathbb{R}}
\newcommand{\C}{\mathbb{C}}
\newcommand{\T}{\mathbb{T}}
\newcommand{\N}{\mathbb{N}}
\newcommand{\x}{\bm x}
\newcommand{\ellinf}{\ell^\infty}
\newcommand{\cv}[1]{\underline{#1}}
\newcommand{\wt}[1]{\widetilde{#1}}
\newcommand{\wh}[1]{\widehat{#1}}
\newcommand{\cal}[1]{{\mathcal #1}}
\newcommand{\ii}{{\rm i}}

\renewcommand{\Re}{\operatorname{Re}}
\renewcommand{\Im}{\operatorname{Im}}
\newcommand{\conj}{\operatorname{conj}}
\newcommand{\Res}{\operatorname{Res}}

\theoremstyle{plain} 
\newtheorem{theorem}{Theorem}[section]
\newtheorem{lemma}[theorem]{Lemma}
\newtheorem{claim}[theorem]{Claim}
\newtheorem{corollary}[theorem]{Corollary}
\newtheorem{proposition}[theorem]{Proposition}

\theoremstyle{definition} 
\newtheorem{definition}[theorem]{Definition}
\newtheorem{hypothesis}[theorem]{Hypothesis}

\theoremstyle{remark}
\newtheorem{remark}{Remark}[section]
\newtheorem{example}{Example}

\begin{document}

\title[Stacking disorder]{Stacking disorder in periodic minimal surfaces}

\author{Hao Chen}
\address[Chen]{Georg-August-Universit\"at G\"ottingen, Institut f\"ur Numerische und Angewandte Mathematik}
\email{h.chen@math.uni-goettingen.de}
\thanks{H.\ Chen is supported by Individual Research Grant from Deutsche Forschungsgemeinschaft within the project ``Defects in Triply Periodic Minimal Surfaces'', Projektnummer 398759432.
M.\ Traizet is supported by the ANR project Min-Max (ANR-19-CE40-0014)}

\author{Martin Traizet}
\address[Traizet]{Institut Denis Poisson, CNRS UMR 7350, Faculté des Sciences et Techniques, Université de Tours }
\email{martin.traizet@lmpt.univ-tours.fr}


\date{\today}

\begin{abstract}
	We construct 1-parameter families of non-periodic embedded minimal surfaces
	of infinite genus in $T \times \R$, where $T$ denotes a flat 2-tori.  Each of
	our families converges to a foliation of $T \times \R$ by $T$.  These
	surfaces then lift to minimal surfaces in $\R^3$ that are periodic in
	horizontal directions but not periodic in the vertical direction.  In the
	language of crystallography, our construction can be interpreted as
	disordered stacking of layers of periodically arranged catenoid necks.  Limit
	positions of the necks are governed by equations that appear, surprisingly,
	in recent studies on the Mean Field Equation and the Painlev\'e VI Equation.
	This helps us to obtain a rich variety of disordered minimal surfaces.  Our
	work is motivated by experimental observations of twinning defects in
	periodic minimal surfaces, which we reproduce as special cases of stacking
	disorder.
\end{abstract}

\maketitle

\section{Introduction}

\subsection{Background}

Triply periodic minimal surfaces (TPMSs) is a topic of trans-disciplinary
interest.  On the one hand, the mathematical notion has been employed to model
many structures in nature (e.g.  biological membrane) and in laboratory (e.g.\
lyotropic liquid crystals); we refer the readers to the book~\cite{hyde1996}
for more information.  On the other hand, natural scientists have been
contributing with important mathematical discoveries, many long precede the
rigorous mathematical treatment.  Examples include the famous gyroid discovered
in~\cite{schoen1970} and proved in~\cite{kgb1996}, as well as its deformations
discovered in~\cite{fogden1993, fogden1999} and recently proved
in~\cite{chen2019}.

The current paper is another example in which mathematics is inspired by
natural sciences.  In~\cite{han2011}, mesoporous crystals exhibiting the
structure of Schwarz' D surface are synthesized.  Remarkably, a twinning
structure, which looks like two copies of Schwarz' D surface glued along a
reflection plane, is observed.  In other word, the periodicity is broken in the
direction orthogonal to the reflection plane.  Thereafter, many other crystal
defects are experimentally observed in more TPMS structures, leading to a
growing demand of mathematical understanding.

Recently, the first named author~\cite{chen2018} responded to this demand with
numerical experiments in Surface Evolver~\cite{brakke1992}.  More specifically,
periodic twinning defects are numerically introduced into rPD surfaces (see
Figure~\ref{fig:twinexample}) and the gyroid.  Success of these experiments
provide strong evidences for the existence of single twinning defects.

\medskip

Moreover, he also became aware of the node-opening techniques developed by the
second named author \cite{traizet2002}.  The idea is to glue catenoid necks
among horizontal planes.  When the planes are infinitesimally close, the necks
degenerate to singular points termed nodes.  It is proved that, if the limit
positions of the nodes satisfy a balancing condition and a non-degeneracy
condition, then it is possible to push the planes a little bit away from each
other, giving a 1-parameter family of minimal surfaces along the way.  The
technique has been used to construct TPMSs~\cite{traizet2008} by gluing necks
among finitely many flat tori, and non-periodic minimal surfaces with
infinitely many planar ends~\cite{morabito2012} by gluing necks among
infinitely many Riemann spheres.

In this paper, we combine the techniques in~\cite{traizet2008}
and~\cite{morabito2012} to glue necks among infinitely many flat tori.  Then
each balanced and non-degenerate arrangement of nodes gives rise to a
1-parameter families of minimal surfaces.  Seen in $T \times \R$, each of these
family converges to a foliation of $T \times \R$ by $T$.  Seen in $\R^3$, the
minimal surfaces are periodic in two independent horizontal directions but not
periodic in any other independent direction.

Our motivation is to rigorously construct twinning defects, but the examples
produced by our construction is far richer.  In the language of
crystallography, our construction can be seen as stacking layers of
periodically arranged catenoid necks.  In the case that $T$ is the 60-degree
torus, for example, we will see that any bi-infinite sequence of 5 stacking
patterns gives arise to a 1-parameter family of minimal surfaces.  These are
then uncountably many families.  In particular, a twinning defect arises from a
stacking fault, which is not periodic but still quite ordered from a physics
point of view.  But most of our examples does not exhibit any order, hence
should be considered as stacking disorders.

\medskip

Back to the twinning, experiments and simulations have shown that TPMSs with
twinning defects decay exponentially to the standard TPMSs.  We will provide
mathematical proof to this physics phenomenon, hence finally justify the term
``TPMS twinning''.  More specifically, we will prove that if a configuration is
eventually periodic, then the corresponding minimal surface is asymptotic to a
TPMS.  The proof uses weighted Banach space as in~\cite{traizet2013}.

\medskip

Our construction uses Implicit Function Theorem, hence only works near the
degenerate limit of foliations, which is not physically plausible.  However,
physicists have proposed formation mechanisms for TPMSs in nature and in
laboratory (e.g.~\cite{charitat1997, michalet1994, conn2006, tang2015}) that
are very similar to node-opening, some even with experiment evidences.  Hence
we may hope that some of the minimal surfaces constructed in this paper,
including those with stacking disorders, would be one day observed in
laboratory.

\begin{figure}[htb!]
	\includegraphics[height=0.4\textheight]{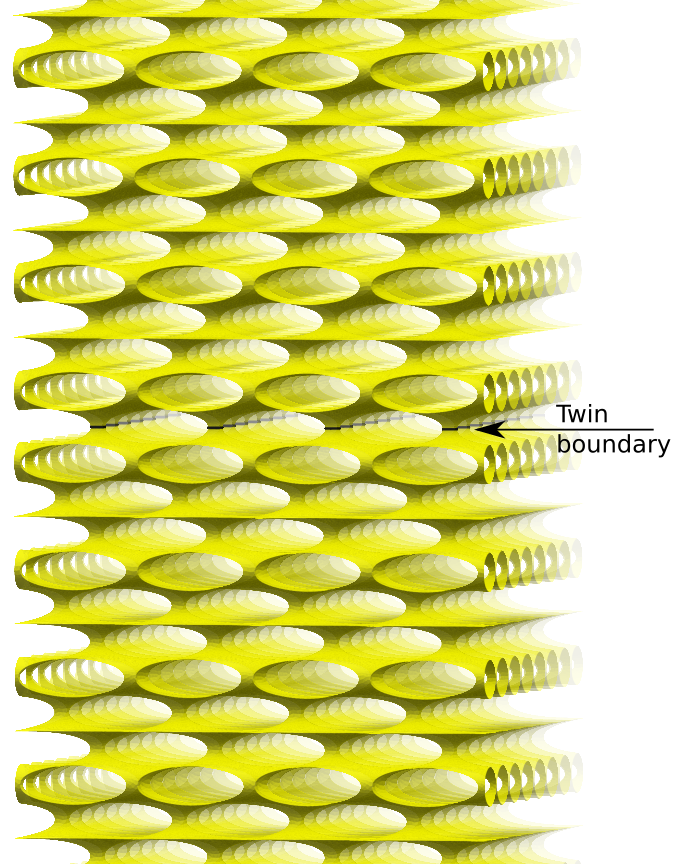}
	\caption{
		Twinning defects in an rPD surface near the catenoid limit, as described in
		Example~\ref{ex:HrPD}.  This is actually an approximation by a TPMS with
		large vertical period.  The surface has a horizontal symmetry plane in the
		middle.  The image was computed in Surface Evolver~\cite{brakke1992} using
		the procedure in \cite{chen2018}.
	}
	\label{fig:twinexample}
\end{figure}

\subsection{Mathematical setting}

A doubly periodic minimal surface (DPMS) $M$ is invariant by two independent
translations, which we may assume to be horizontal.  Let $\Gamma$ be the
two-dimensional lattice generated by these translations, then $M$ projects to a
minimal surface $M/\Gamma$ in $\R^3/\Gamma = T \times \R$, where
$T=\R^2/\Gamma$ denotes a flat 2-torus.  Immediate examples of infinite genus
are given by triply periodic minimal surfaces (TPMSs), if one ignores one of
their three periods.  Motivated by experimental observations mentioned above,
we are particularly interested in non-periodic DPMS with infinite genus.

A flat torus in $T \times \R$ is \emph{horizontal} if it has the form $T
\times \{h\}$ for some $h \in \R$; then $h$ is called the \emph{height} of the
torus.  Informally speaking, we construct minimal surfaces that look like
infinitely many horizontal flat tori in $T \times \R$, ordered by increasing
height, with one catenoid neck between each adjacent pair.  The tori are then
labeled by $k \in \Z$ in the order of height.  The catenoid necks are also
labeled by $k \in \Z$, such that the $k$-th neck is between the $k$-th and the
$(k+1)$-th tori.

\begin{remark}
	Our construction can, in principle, handle finitely many necks between each
	adjacent pair of tori.  But in view of the immediate interest from material
	sciences, we will only glue one catenoid neck between each adjacent tori.
	This also eases the notations and facilitates the proofs, but still produces
	a rich variety of examples.
\end{remark}

More formally, we say that a minimal surfaces $M \in T \times \R$ is
\emph{stacked} if there is an increasing sequence of real numbers $(h_k)_{k
\in \Z}$ satisfying
\begin{itemize}
	\item $M \cap (T \times \{h_k\})$ has a single connected component that
		projects to a null-homotopic smooth simple closed curve in $T$;

	\item $M \cap (T \times (h_k, h_{k+1}))$ is homeomorphic to $T$ with
		two disks removed.
\end{itemize}
Then the $k$-th neck can be interpreted as an annular neighborhood of $M \cap
(T\times\{h_k\})$.

\begin{remark}
	The term ``stacked'' is borrowed from crystallography.  Closed-packed
	structures are often described as a result of stacking layers of periodically
	arranged atoms, one on top of another.  Analogously, a stacked minimal
	surface can be seen as obtained by stacking layers of periodically arranged
	catenoid necks.
\end{remark}

We intend to construct 1-parameter families $M_t$, $t>0$, of stacked minimal
surfaces such that, in the limit $t \to 0$, every neck converges to a catenoid
after suitable rescaling.  If we rescale to keep the size of the torus, then
$M_t$ will converge to a foliation of $T \times \R$ by $T$, and necks
converge to singular points, which we call \emph{nodes}.

\subsection{Definitions and main result}

\begin{definition}
	A \emph{node configuration} is a sequence $\bm p=(p_k)_{k \in \Z}$ such that
	$p_k \in T$ for all $k \in \Z$.
\end{definition}

We use $p_k$ to prescribe the limit position of the $k$-th node, and assume
that

\begin{hypothesis}[Uniform separation]\label{hyp:separate}
	There exists a constant $\varrho > 0$ such that $p_k$ and $p_{\ell}$ are at least
	at distance $\varrho$ apart whenever $|k-\ell | = 1$.
\end{hypothesis}

Assume that $2\omega_1=1$ and $2\omega_2=\tau$ generate the lattice $\Gamma$,
so $T = T_\tau = \C / (\Z + \tau \Z)$.  Without loss of generality, we also
assume that $\Im\tau > 0$.  Then for $p \in T$, we use $x(p;\tau)$ and
$y(p;\tau)$ to denote its coordinates in the basis $1$ and $\tau$, and define
the function
\[
	\xi(p;\tau) = x(p;\tau) \eta_1(\tau) + y(p;\tau) \eta_2(\tau),
\]
where $\eta_i(\tau) = \zeta(z+2\omega_i;\tau)-\zeta(z;\tau) = 2\zeta(\omega_i;\tau)$ for
$i = 1, 2$, and
\[
	\zeta(z;\tau) = \frac{1}{z} \sum_{0 \ne u \in \Z+\tau\Z} \Big(\frac{1}{z-u} + \frac{1}{u} + \frac{z}{u^2}\Big)
\]
is the Weierstrass zeta function associated to $\Gamma$.

The following definitions are borrowed from \cite{traizet2008} and
\cite{morabito2012}.  Given a node configuration $\bm p$, the \emph{force}
$F_k$ exerted on the node $p_k$ by other nodes is
\begin{equation} \label{eq:force1}
	F_k :=
	\zeta(p_{k+1}-p_k;\tau) + \zeta(p_{k-1}-p_k;\tau)
	+ 2\xi(p_k;\tau) - \xi(p_{k+1};\tau) - \xi(p_{k-1};\tau)
\end{equation}

\begin{definition}
	A node configuration is said to be \emph{balanced} if $F_k=0$ for all $k \in \Z$.
\end{definition}

Since our construction uses the Implicit Function Theorem, we need the
differential of the force to be invertible in some sense.
As explained in \cite{morabito2012}, $(p_k)$ is not the right parameter to formulate non-degeneracy
and one needs to introduce the sequence $\bm q=(q_k)$ defined by
\[
	q_k=p_k-p_{k-1}.
\]
Note that the uniform separation hypothesis can be reformulated as $(q_k)$
being bounded away from $0$.  From now on, we use the term ``configuration''
for the infinite sequence $(q_k)$.

Under the new variables, \eqref{eq:force1} becomes
\begin{equation}
	F_k
	:= \zeta(q_{k+1};\tau) - \zeta(q_k;\tau)
	+ \xi(q_k;\tau) - \xi(q_{k+1};\tau)\nonumber
	= G_{k+1} - G_k.\label{eq:force2}
\end{equation}
where
\[
	G_k = G(q_k;\tau) = \zeta(q_k;\tau) - \xi(q_k;\tau).
\]
A configuration is then balanced if $(G_k)$ is a constant sequence, i.e.\ $G_k
= G_0$ for all $k \in \Z$.

\begin{definition}
	A configuration is said to be \emph{non-degenerate} if the differential of
	$(G_k)_{k \in \Z}$ with respect to $(q_k)_{k \in \Z}$, as a map from
	$\ellinf$ to itself, is an isomorphism.
\end{definition}

Note that $G(q;\tau)$ is periodic in $q$ but not meromorphic.  The function
$G(q;\tau)$ is called the \emph{Hecke form} in~\cite{lang1995}.  It was proved
by Hecke~\cite{hecke1927} that, if $q = (k_1+k_2\tau)/N$ with
$\gcd(k_1,k_2,N)=1$, then $G(q;\tau)$ is a modular form of weight $1$ with
respect to the congruence group $\Gamma(N)$.  Recently, the Hecke form gained
popularity for its importance in the study of PDEs, including the Mean Field
Equation, the Painlev\'e VI Equation, and the (generalized) Lam\'e Equation;
see~\cite{lin2016} for a survey.  We will exploit some of the recent
results~\cite{lin2010, chenzj2018, bergweiler2016} in our construction.

\medskip

Now we are ready to state our main theorems

\begin{theorem}
	If a configuration $\bm q$ is balanced, non-degenerate, and satisfies the
	uniform separation hypothesis, then there exists in $T \times \R$ a
	1-parameter family $(M_t)_{0<t<\epsilon}$ of embedded stacked minimal
	surfaces which, in the limit $t \to 0$, converges to a foliation of $T \times
	\R$ by $T$.  Moreover, the necks have asymptotically catenoidal shape and their
	limiting positions in $T$ are prescribed by $\bm p$.
\end{theorem}

\begin{theorem} \label{thm:asymptoticTPMS}
	Let $\bm q$ and $\bm q'$ be two balanced and non-degenerate configurations
	that satisfy the uniform separation hypothesis.  Assume that $(q_k)$ is
	periodic (in the sense $q_{k+N}=q_k$) and $q'_k = q_k$ for all $k \ge 0$.
	Let $(M_t)$ and $(M'_t)$ denote the corresponding 1-parameter families of
	minimal surfaces.  Then $M_t$ is a TPMS and $M'_t$ is asymptotic to a
	translation of $M_t$ as $x_3\to\infty$.
\end{theorem}

\medskip

The paper is organized by increasing technicality.  After reviewing some
examples in Section~\ref{sec:examples}, we prove our main theorems in
Section~\ref{sec:proof} and Section~\ref{sec:asymptotic-immersion},
respectively.  Technical ingredients of the proofs are delayed to later
sections.  In Section~\ref{sec:omega} we prove the existence and smooth
dependence on parameters of a holomorphic 1-form $\omega$ that we used in the
Weierstrass data in Section~\ref{sec:proof}.  In Section~\ref{sec:asymptotic},
we study the asymptotic behavior of $\omega$ and other parameters, which is
crucial for proving the TPMS asymptotic behavior in
Section~\ref{sec:asymptotic-immersion}.

\section{Examples}
\label{sec:examples}

Given an infinite sequence of planes, if only one node is opened between each
adjacent pair, the result is necessarily a Riemann minimal
example~\cite{morabito2012}.  We now show that opening nodes among flat tori is
a sharp contrast.  Although we only open one node between each adjacent pair of
tori, we still obtain a rich variety of balanced configurations.  We produce
configurations using the following

\begin{proposition}\label{prop:allconf}
	Let $\wt{q}_0, \wt{q}_1, \cdots, \wt{q}_{n-1}$ be $n$ solutions of the equation
	\begin{equation}\label{eq:F}
		G(q;\tau)=C
	\end{equation}
	for the same complex constant $C$. Assume that the differential of $G$ with
	respect to $q$ at $\wt{q}_k$, as a function from $\R^2$ to $\R^2$, is
	non-singular for each $k$.  Then any bi-infinite sequence $(q_k)_{k\in\Z}$ of
	elements in the set $\{\wt{q}_0, \wt{q}_1, \cdots, \wt{q}_{n-1}\}$ is a
	balanced, non-degenerate configuration satisfying the uniform separation
	hypothesis.
\end{proposition}

\begin{proof}
	The configuration is balanced by definition. Since there is only a finite
	number of points $\wt{q}_k$, the uniform separation hypothesis is satisfied and
	the differentials of $G$ with respect to $q$ at $\wt{q}_k$, as well as their
	inverses, are uniformly bounded. Then the differential of $(G_k)$ with
	respect to $(q_k)$ is clearly an automorphism of $\ell^{\infty}$.
\end{proof}
\medskip

We can produce a rich variety of examples thanks to the fact that \eqref{eq:F}
often has several solutions, which we can combine in any arbitrary way to form
stacking disorders. We first discuss the number of solutions of Equation
\eqref{eq:F}.

\subsection{Solutions with \texorpdfstring{$C=0$}{C=0}}
\label{section:Czero}

The solutions to $G(q;\tau)=0$ are critical points of the Green function on a
flat torus. The number of critical points and their non-degeneracy has been
investigated in \cite{lin2010, chenzj2018, bergweiler2016}.  Recall that the
function $G$ is odd and $\Gamma$-periodic in the variable $q$.  Hence for any
$\tau$, the 2-division points $1/2$, $\tau/2$ and $(1+\tau)/2$ are trivial
solutions to $G(q;\tau)=0$. Moreover, it is recently proved~\cite{chenzj2018}
(see also~\cite{bogdanov2017}) that all three trivial solutions are
non-degenerate for a generic $\tau$, and at least two of them are
non-degenerate for any $\tau$.  Using Proposition \ref{prop:allconf}, any flat
2-torus $T$ admits uncountably many balanced and non-degenerate configurations,
giving rise to uncountably many 1-parameter families of non-periodic minimal
surfaces in $T\times\R$.

Since $G(q;\tau)$ is odd in $q$, non-trivial solutions of $G(q;\tau)=0$ must
appear in pairs.  Using a deep connection with the mean field equation, Lin and
Wang proved that $G(q;\tau)=0$ has at most one non-trivial solution pair for a
fixed $\tau$~\cite[Theorem~1.2]{lin2010}.  In other words, $G(q;\tau)=0$ has
either three or five solutions.  A direct and simpler proof was later provided
by Bergweiler and Eremenko~\cite{bergweiler2016}, who also give an explicit
criterion distinguishing $\tau$'s with three and five solutions.  Moreover, in
the case of five solutions, all solutions are non-degenerate~\cite{lin2017}.
For an explicit example, with $\tau = \exp(\ii \pi/3)$, the non-trivial pair of
solutions are
\begin{equation} \label{eq:nontrivialsolution}
	q=\pm(1+\tau)/3.
\end{equation}

\subsection{TPMS examples}

For a fixed $\tau$, any bi-infinite sequence of the solutions of $G(q;\tau)=0$
is a balanced configuration, hence gives rise to a family of minimal surfaces.
From a crystallographic point of view, most of these surfaces would be
considered as disordered.

TPMSs with perfect periodic patterns, arising from periodic configurations, are
certainly the most interesting cases for crystallographers.  In the following,
we list some TPMSs of genus three that arise from configurations with period 2
(namely $q_{2k}=q_0$ and $q_{2k+1}=q_1$ for all $k\in\Z$). This completes the
discussion in Section 4.3.3 of \cite{traizet2008} which was incomplete.

\begin{example}\leavevmode
	If $q_1=q_0$ (so $\bm q$ is constant) the configuration is trivially
	balanced.  These configurations give rise to TPMSs in Meeks' 5-parameter
	family.  Some famous examples are:
	\begin{itemize}
		\item $\Re \tau=0$, $q_0 = (1+\tau)/2$, gives an orthorhombic deformation
			family of Schwarz' P surface (named oPa in~\cite{fogden1992}), which
			reduces to Schwarz' tP family when $\tau = \ii$.

		\item $|\tau| = 1$, $q_0 = (1+\tau)/2$, gives another orthorhombic
			deformation family of Schwarz' P surface (named oPb
			in~\cite{fogden1992}), which reduces to Schwarz' tP family when $\tau =
			\ii$.

		\item $\Re \tau=0$, $q_0 = 1/2$, gives an orthorhombic deformation family
			of Schwarz' CLP surface (named oCLP' in~\cite{fogden1992}).

		\item $\tau = \exp(\ii\pi/3)$, $q_0 = (1+\tau)/3$, gives a rhombohedral
			deformation family of Schwarz' D and P surfaces (known as rPD).
	\end{itemize}
\end{example}
Note that the first three examples are obtained from trivial solutions of
$G(q;\tau)=0$, and the fourth one is obtained from the non-trivial solutions
\eqref{eq:nontrivialsolution}.
\begin{example} When $\Re \tau=0$, $q_0=1/2$ and $q_1=\tau/2$, gives
	the newly discovered $o\Delta$ surfaces~\cite{chen2018a}.  Note that this
	example is obtained by alternating two trivial solutions of the equation
	$G(q;\tau)=0$.
\end{example}
\begin{example}
	Examples with $q_1=-q_0$ were studied in~\cite{chen2018b}.  In particular
	\begin{itemize}
		\item $\tau = \exp(\ii\pi/3)$, $q_0 = (1+\tau)/3$, gives hexagonal Schwarz'
			H family.  Note that this example is obtained by alternating the two
			non-trivial solutions \eqref{eq:nontrivialsolution}.

		\item There exists a real number $\pi/2 > \theta^* > \pi/3$ such that
			whenever $\tau = \exp(i \theta)$ with $\theta < \theta^*$, the
			configuration is balanced with $q_0 = c(1+\tau)$ for a unique $c<1/2$.
			This leads to the orthorhombic deformations of Schwarz' H surfaces
			in~\cite{chen2018b}.  Existence and uniqueness of $\theta^*$ was
			essentially proved~\cite{weber2002, weber2009}, and independently
			in~\cite{lin2010}.  Its value was computed in~\cite{chen2018b} explicitly
			as
			\begin{equation}\label{eq:torus70}
				\theta^* = 2\arctan\frac{K'(m)}{K(m)} \approx 1.23409,
			\end{equation}
			where $m$ is the unique solution of $2E(m)=K(m)$, and $K(m)$, $K'(m)$ and
			$E(m)$ are elliptic integrals of the first kind, associated first kind,
			and second kind, respectively.

		\item For general $\tau$, non-trivial $q_0$'s that give balanced
			configurations are studied in~\cite{weber2002, lin2010, chenzj2018} and
			numerically in~\cite{chen2018b}.
	\end{itemize}
\end{example}
\begin{remark}
	For crystallographers, the rPD and the H surfaces are analogous to,
	respectively, the cubic and hexagonal close-packing.
\end{remark}

\begin{remark}
	Interestingly, configurations in Examples 1 and 3 give rise to TPMSs no
	matter their degeneracy.  Those in Example 1 form a 4-parameter family, and
	they are limits of Meeks 5-parameter family.  Those in Example 3 are
	degenerate only if $q_0$ is a 2-division point, hence reduces to Example 1.
	In particular, the degenerate configuration with $\tau = \exp(i \theta^*)$
	and $q_0=q_1=(1+\tau)/2$ is considered in \cite{chen2018b}.  It is the limit
	of a 1-parameter family of TPMSs that are degenerate in the sense that the
	same deformation of the lattices may lead to different deformations of the
	TPMSs.  It is not clear to what extent does this phenomenon generalize.
\end{remark}

\subsection{Examples of TPMSs with defects}

From the TPMS examples above, we obtain the following examples with asymptotic
TPMS behavior by Theorem \ref{thm:asymptoticTPMS}.  From a crystallographic
point of view, they are TPMSs with planar defects.

\begin{example}
	We may combine oPa, oCLP' and o$\Delta$ surfaces using the trivial solutions
	of $G(q;\tau)=0$ at the 2-division points.  For example:
	\begin{itemize}
		\item  The configuration $(q_k)_{k\in\Z}$ defined by
			\[
				\Re\tau=0,\qquad q_k=\begin{cases}
					1/2&\text{if $k<0$,}\\
					(1+\tau)/2&\text{if $k\geq 0$,}
				\end{cases}
			\]
			gives rise to non-periodic minimal surfaces in $T\times\R$ which are
			asymptotic to oPa surfaces as $x_3\to+\infty$ and oCLP' surfaces as
			$x_3\to -\infty$.

		\item  The configuration $(q_k)_{k\in\Z}$ defined by
			\[
				\Re\tau=0,\qquad q_k=\begin{cases}
					1/2&\text{if $k<0$,}\\
					\tau/2&\text{if $k\geq 0$,}
				\end{cases}
			\]
			gives rise to non-periodic minimal surfaces in $T\times\R$ which are
			asymptotic, as $x_3\to+\infty$ and $x_3\to -\infty$, to two different
			oCLP' surfaces.  When $\tau = \ii$, the two oCLP' surfaces differ only by
			a 180-degree rotation with horizontal axis, hence can be seen as a
			rotation twin.

		\item  The configuration $(q_k)_{k\in\Z}$ defined by
			\[
				\Re\tau=0,\qquad q_k=\begin{cases}
					1/2&\text{if $k<0$ and odd,}\\
					\tau/2&\text{if $k<0$ and even,}\\
					(1+\tau)/2&\text{if $k\geq 0$,}
				\end{cases}
			\]
			gives rise to non-periodic minimal surfaces in $T\times\R$ which are
			asymptotic to oPa surfaces as $x_3\to+\infty$ and the newly discovered
			o$\Delta$ surfaces~\cite{chen2018a} as $x_3\to -\infty$.
	\end{itemize}

	We certainly did not list all possible combinations.  Note that these
	examples generalize to other $\tau$'s in an obvious way, giving rise to
	non-periodic minimal surfaces that are asymptotic to unnamed TPMSs.
\end{example}

\begin{example}\label{ex:HrPD}
	We may combine H and rPD surfaces using the pair of non-trivial
	solutions~\eqref{eq:nontrivialsolution}.  For example:

	\begin{itemize}
		\item The configuration defined by
			\[
				\tau=\exp(\ii\pi/3),\qquad q_k=\begin{cases}
					(1+\tau)/3&\text{if $k<0$,}\\
					-(1+\tau)/3&\text{if $k\geq 0$,}
				\end{cases}
			\]
			gives rise to non-periodic minimal surfaces in $T\times\R$ which are
			symptotic,  as $x_3\to+\infty$ and $x_3\to -\infty$, to two Schwarz rPD
			surfaces that differ only by a reflection, hence are twins of Schwarz
			rPD-surfaces (see Figure \ref{fig:twinexample}). Such a D-twin has been
			observed experimentally.

		\item The configuration defined by
			\[
				\tau=\exp(\ii\pi/3),\qquad q_k=\begin{cases}
					(1+\tau)/3&\text{if $k<0$ and even,}\\
					-(1+\tau)/3&\text{otherwise,}
				\end{cases}
			\]
			gives rise to non-periodic minimal surfaces in $T\times\R$ which are
			asymptotic to Schwarz rPD surfaces as $x_3 \to +\infty$ and Schwarz' H
			surfaces as $x_3 \to -\infty$.

		\item The configuration defined by
			\[
				\tau=\exp(\ii\pi/3),\qquad q_k=\begin{cases}
					(1+\tau)/3&\text{if $k<0$ and even or if $k>0$ and odd,}\\
					-(1+\tau)/3&\text{otherwise,}
				\end{cases}
			\]
			gives rise to non-periodic minimal surfaces in $T\times\R$ which are
			asymptotic, as $x_3\to+\infty$ and $x_3\to -\infty$, to two different
			Schwarz' H surfaces that differ only by a horizontal translation.
	\end{itemize}

	We certainly did not list all possible combinations.  These examples
	generalize, in an obvious way, to any other $\tau$'s such that $G(q;\tau)=0$
	has a pair of non-trivial solutions, giving rise to non-periodic minimal
	surfaces that are asymptotic to unnamed TPMSs.
\end{example}

\subsection{Historical remarks}

The Hecke form $G(q;\tau)$ has been studied independently by the PDE and
minimal surface communities.  Hence we would like to point out some connections
between their approaches.

\medskip

Solutions to $G(q;\tau)=0$ are particularly interesting as they are the
critical points of the Green function on a flat torus $T_\tau$~\cite{lin2010,
chenzj2018}.  This is no surprise in the context of node-opening construction
of TPMSs.  In~\cite{traizet2008}, the forces between nodes are compared to
electrostatic forces between electric charges.  The Green function is nothing
but the potential function of the electric field generated by periodically
arranged charges.  The balancing condition asks that all charges are in
equilibrium, hence at a critical point of the potential.

A 2-division point $\omega$ is degenerate if
\[
	\frac{\tau\wp(\omega;\tau) + \eta_2(\tau)}{\wp(\omega;\tau)+\eta_1(\tau)}
\]
is real.  This is the quotient of periods of the elliptic function $\wp(z;\tau)
- \wp(\omega;\tau)$.  So if $\omega$ is degenerate, the torus $T_\tau$ admits a
meromorphic 1-form with a double pole, a double zero, and only real periods.
Such tori are no stranger to the minimal surface theory.  In particular, the
unique rhombic torus with period quotient $-1$ was used to construct helicoids
with handles~\cite{weber2002, weber2009}, and its angle has an explicit
expression as given in~\eqref{eq:torus70} (see~\cite{chen2018b}).  On the PDE
side, existence of this torus was independently proved in~\cite{lin2010}.

Following~\cite{weber2009}, we propose a simple construction for the torus and
the 1-form: Slit the complex plane along the real segment $[-1,1]$.  Identify
the top edge of $[-1,-x]$ (resp.\ $[-x,1]$) with the bottom edge of $[x,1]$
(resp.\ $[-1,x]$), where $x \in [0,1)$ and $(x+1)/(x-1)$ is the quotient of
periods.  The result is a torus carrying a cone metric with two cone
singularities, one of cone angle $6\pi$ at the point identified with $\pm 1$
and $\pm x$, the other of cone angle $-2\pi$ at $\infty$.  Its periods are
obviously real.  The same torus with flat metric is $T_\tau$.

It follows easily from~\cite{weber2002} that there exists a unique torus for
each real period quotient.  This essentially proves Theorem~6.1(1)
in~\cite{chenzj2018}.

\subsection{Solutions with \texorpdfstring{$C\neq 0$}{C != 0}}
\label{section:Cnotzero}

As said before, \cite{bergweiler2016} proved again that the equation
$G(z;\tau)=0$ has either three or five solutions.  Their elegant argument can
be adapted to the case $C\neq 0$ and yields the following result:

\begin{theorem} \label{thm:Cnotzero}
	For given $\tau$ and $C\in\C$, the equation $G(q;\tau)=C$ has at least 1 and
	at most 5 solutions.
\end{theorem}

\begin{proof}
	We adapt the argument of \cite{bergweiler2016} to the case $C\neq 0$.
	First of all, following \cite{bergweiler2016}, we write
	\[
		G(z;\tau)=\zeta(z;\tau)+az+b\overline{z}\quad\mbox{ with }
		a=\frac{\pi}{\Im(\tau)}-\eta_1\quad\mbox{ and }\quad
		b=-\frac{\pi}{\Im(\tau)},
	\]
	and define the anti-meromorphic function $g$ by
	\[
		g(z)=-\frac{1}{b}\left(\overline{\zeta(z)}+a\overline{z}-\overline{C}\right)=z-\frac{1}{b}\left(\overline{G(z;\tau)}-\overline{C}\right).
	\]
	The only difference with \cite{bergweiler2016} is that $g$ is not odd anymore
	if $C\neq 0$.

	It is proved in Lemma~4 of \cite{bergweiler2016} using complex dynamics that
	$g$ has at most two attracting fixed points modulo $\Gamma$.  The proof
	carries over to the case $C\neq 0$ with no change.  The function $g$
	satisfies $g(z+\omega)=g(z)+\omega$ for all $\omega\in\Gamma$.  Hence we may
	define a map $\phi:\C/\Gamma\to\C\cup\{\infty\}$ by $\phi(z)=z-g(z)$.  The
	equation $G(z;\tau)=C$ is equivalent to $\phi(z)=0$.  Since $\phi$ has a
	single simple pole, where the differential reverses orientation, its degree
	(as a map between compact manifolds of the same dimension) is $-1$.  Hence
	the equation $\phi(z)=0$ always has at least one solution.  In fact,
	$\phi(z)=0$ has exactly one solution if $|C|$ is sufficiently large.

	We have $\det(d\phi)=1-|\overline{\partial}g|^2$.  If $0$ is a regular value
	of $\phi$, then writing $N^+$ and $N^-$ for the number of zeros of $\phi$
	with respectively positive and negative determinant of $d\phi$, we see that
	$N^+$ is the number of attracting fixed points of $g$, so $N^+\leq 2$. Then
	since $\deg(\phi)=-1$, we have $N^-=N^++1$ so the total number of zeros of
	$\phi$ is $\leq 5$.  If $0$ is a critical value of $\phi$, the number of
	zeros of $\phi$ is still less than 5 by the same argument as in
	\cite{bergweiler2016}, Lemma~5.

	Observe that if $C=0$, then $\phi$ is odd so has the three half-lattice
	points as trivial zeros: this is the only place in this part of the argument
	of \cite{bergweiler2016} where the parity of $\phi$ is really used.
\end{proof}

Note that to apply Theorem~\ref{thm:Cnotzero} to minimal surfaces, we still
need to study the non-degeneracy of the solutions.

\section{Construction} \label{sec:proof}

\subsection{Parameters}

The parameters of the construction are a real number $t$ in a neighborhood
of~$0$ and four sequences of complex numbers
\[
	\bm a=(a_k)_{k \in \Z}, \quad \bm b=(b_k)_{k \in \Z}, \quad \bm v=(v_k)_{k \in \Z}, \quad \text{and} \quad \bm \tau=(\tau_k)_{k \in \Z}
\]
in $\ellinf$.  Each parameter is in a small $\ellinf$-neighborhood of a central
value denoted with an underscore.  We will calculate that the central value of
the parameters are:
\begin{equation} \label{eq:central-value}
	\cv{a}_k=-\frac{1}{2},\quad
	\cv{b}_k=\frac{1}{2}\xi(\cv{v}_k;\cv{\tau}_k),\quad
	\cv{v}_k=(-\conj)^kq_k,\quad
	\cv{\tau}_k=(-\conj)^k\tau,
\end{equation}
where $\conj$ denotes conjugation, $\tau$ and $(q_k)$ prescribe the flat
2-torus and the configuration as in the introduction.  We use $\x=(\bm a, \bm
b, \bm v, \bm\tau)$ to denote the vector of all parameters but $t$.  When
required, the dependence of objects on parameters will be denoted with a
bracket as in $\Sigma[t,\x]$, but will be omitted most of the time.

\subsection{Opening nodes and the Gauss map} 
\label{sec:opening-nodes}

We denote by $T_k=T_k[\x]$ the torus $\C/(\Z+\tau_k\Z)$. The point $z=0$ in
$T_k$ is denoted by $0_k$.  We define the elliptic function $g_k=g_k[\x]$ on
$T_k$ by
\[
	g_k(z) = a_k \big( \zeta(z;\tau_k) - \zeta(z-v_k;\tau_k) \big) + b_k.
\]
It has two simple poles at $0_k$ and $v_k$, with residues $a_k$ and $-a_k$,
respectively.  Observe that $1/g_k$ is a local complex coordinate in a
neighborhood of $0_k$ and $v_k$.  Hence for a sufficiently small
$\varepsilon>0$, $1/g_k$ gives a diffeomorphism $z^+_k$ from a neighborhood of
$v_k$ in $T_k$ to the disk $D(0,2\varepsilon) \subset \C$, and a diffeomorphism
$z^-_k$ from a neighborhood of $0_{k}$ in $T_k$ to the disk
$D(0,2\varepsilon)$.  Provided that $\x$ is sufficiently close to $\cv{\x}$ and
by Hypothesis~\ref{hyp:separate}, $\varepsilon$ can be chosen independent of
$k$ and $\x$.  Let $D^\pm_k$ be the disk $|z^\pm_k| < \varepsilon$ in $T_k$.

\medskip

Consider the disjoint union of all $T_k$ for $k\in\Z$.  If $t=0$, identify $v_k
\in T_k$ and $0_{k+1} \in T_{k+1}$ to create a node. The resulting Riemann
surface with nodes is denoted $\Sigma[0,\x]$.  If $t\neq 0$ and
$|t|<\varepsilon$, then for each $k \in \Z$, remove the disks $|z^\pm_k| \le
t^2/\varepsilon$ from $T_k$, and let $A^\pm_k$ be the annuli $t^2/\varepsilon <
|z^\pm_k|< \varepsilon$.  Identify $A_k^+$ and $A_{k+1}^-$ by $z_k^+ z_{k+1}^-
= t^2$.  This opens nodes and creates a neck between $T_k$ and $T_{k+1}$.  The
resulting Riemann surface is denoted $\Sigma=\Sigma[t,\x]$.

If $t \neq 0$, we define the Gauss map $g=g[t,\x]$ explicitly on $\Sigma[t,\x]$
by
\[
	g(z)=(tg_k(z))^{(-1)^{k+1}} \quad \text{in $T_k$.}
\]
Then $g$ takes the same value at the points that are identified when defining
$\Sigma$.  So $g$ is a well-defined meromorphic function on $\Sigma$.

\subsection{Height differential} \label{ssec:dh}

We define
\[
	\Omega_k:=T_k\setminus\Big(\overline{D_k^+}\cup \overline{D_k^-}\Big) \quad \text{and} \quad \Omega:=\bigsqcup_{k\in\Z}\Omega_k\subset\Sigma.
\]
All circles $\partial D_k^-$ are homologous in $\Sigma$. This homology class is
denoted $\gamma$.  We denote by $\alpha_k$ and $\beta_k$ the standard
generators of the homology of $T_k$, namely the homology classes of $[0,1]$ and
$[0,\tau_k]$ modulo $\Z+\tau_k\Z$.  We choose representatives of $\alpha_k$ and
$\beta_k$ within $\Omega_k$, so they can be seen as curves on $\Sigma$.

\medskip

By Proposition \ref{prop:omega} in Section \ref{sec:omega}, for $t$ small
enough, there exists a holomorphic 1-form $\omega=\omega[t,\x]$ on
$\Sigma[t,\x]$ with imaginary periods on $\alpha_k$, $\beta_k$ for all $k\in\Z$
and $\int_{\gamma}\omega=2\pi\ii.$ We define the height differential $dh$ by
\[
	dh=t\omega.
\]
If $t=0$, $\omega$ is allowed to have simple poles at the nodes (a so-called
\emph{regular 1-form} on a Riemann surface with nodes).  So $\omega$ has simple
poles at $0_k$ and $v_k$, with residues $1$ and $-1$ respectively, and
imaginary periods on $\alpha_k$ and $\beta_k$.  By
Proposition~\ref{prop:omega}, we have explicitly
\begin{equation} \label{eq:omega0}
	\omega[0,\x]=\left(\zeta(z;\tau_k)-\zeta(z-v_k;\tau_k)-\xi(v_k;\tau_k)\right)dz\quad\text{in $T_k$}.
\end{equation}

\medskip

Finally, $\omega[t,\x]$ restricted to $\Omega$ depends smoothly on $(t,\x)$ in
a sense which we now explain.

Some care is required because the domain $\Omega$ depends on the parameters.
To formulate the smooth dependence, we pullback $\omega$ to a fixed domain as
follows.  Let $\T=\C/(\Z+\ii\Z)$ be the standard square torus. Let $\psi_k[\x]$
be the diffeomorphism defined by
\[
	\psi_k[\x]:\T\to T_k[\x],\quad \psi_k(x+\ii y)=x+\tau_k y.
\]
Let $\wt{v}_k=\psi_k^{-1}(v_k)$.
Fix a small $\varepsilon'>0$ and define
\[
	\wt{\Omega}_k = \T \setminus \big( \overline{D(0,\varepsilon')} \cup
	\overline{D(\cv{\wt{v}}_k,\varepsilon')} \big)
	\quad \text{and} \quad
	\wt{\Omega}=\bigsqcup_{k\in\Z}\wt{\Omega}_k.
\]
If $\varepsilon'$ is small enough and $\x$ is close enough to $\cv{\x}$, we
have $\Omega_k\subset\psi_k(\wt{\Omega}_k)$.  Moreover, if $t$ is small enough,
the disks $|z_k^{\pm}|<t^2/\varepsilon$ which were removed when opening nodes
are outside $\psi_k(\wt{\Omega}_k)$, so we can see $\psi_k(\wt{\Omega}_k)$ as a
domain in $\Sigma$.  We define $\psi:\wt{\Omega}\to\Sigma$ by $\psi=\psi_k$ on
$\wt{\Omega}_k$.  Then $\psi^*\omega$ is a smooth 1-form on $\wt{\Omega}$.
(Note that $\psi^*\omega$ is not holomorphic, because $\psi_k$ is not
conformal.)

\medskip

We define the pointwise norm of a (not necessarily holomorphic) 1-form $\eta$
on a domain $U$ in $\C$ or a torus $\C/\Gamma$ by $|\eta(z)|=\sup_{X\in\C^*}
\frac{|\eta(z)X|}{|X|}$. We denote $C^0(U)$ the Banach space of
1-forms $\eta=f_1 \,dx+f_2 \,dy$ with $f_1$, $f_2$ bounded continuous functions
on $U$, with the sup norm.
We can now state:
\begin{proposition} \label{prop:omega-smooth}
	The map $(t,\x)\mapsto \psi^*\omega[t,\x]$ is smooth from a neighborhood of
	$(0,\cv{\x})$ to $C^0(\wt{\Omega})$.
\end{proposition}
This is the content of Proposition \ref{prop:omega}(c).

\begin{remark}
	The point here is that $\wt{\Omega}$ is a fixed domain, independent of the
	parameters.  The domain $\psi(\wt{\Omega})\subset\Sigma$ depends on
	$(\tau_k)$ but not on $t$ nor the other parameters.  It contains the domain
	$\Omega$ which fully depends on $\x$.  Since $\tau_k$ is in a neighborhood of
	$\cv{\tau}_k$, we have $c^{-1}\leq |\psi_k^*dz|\leq c$ for some uniform
	constant $c$.  Hence for any holomorphic 1-form $\eta$ on $\Sigma$,
	\[
		c^{-1}\|\psi_k^*\eta\|_{C^0(\wt{\Omega})}\leq \|\eta\|_{C^0(\psi(\wt{\Omega}))}
		\leq c\|\psi_k^* \eta\|_{C^0(\wt{\Omega})}.
	\]
\end{remark}

\subsection{Zeros of the height differential} \label{ssec:zeros}

We use the Weierstrass parametrization
\[
	\Sigma[t,\x] \ni z \mapsto \Re \int_{z_0}^z (\Phi_1, \Phi_2, \Phi_3) \in
	\R^3,
\]
where
\[
	(\Phi_1,\Phi_2,\Phi_3):=\Big(\frac{1}{2}(g^{-1}-g), \frac{\ii}{2}(g^{-1}+g), 1 \Big) dh
\]
are holomorphic differentials, so that the Weierstrass parametrization is an
immersion.  So we need to solve the Regularity Problem, which asks that $dh$
has a zero at each zero or pole of $g$, with the same multiplicity, and no
other zeros.

The elliptic function $g_k[\x]$ has degree 2 so it has two zeros in $T_k$ which
we denote $Z_{k,1}[\x]$ and $Z_{k,2}[\x]$. Since $g_k[\x]$ has poles at $0_k$
and $v_k$, its zeros are in $\Omega_k$ provided that $\varepsilon$ is small
enough.  Note that we cannot rule out the possibility of a double zero
$Z_{k,1}=Z_{k,2}$ at $\cv{\x}$, in which case $Z_{k,1}$ and $Z_{k,2}$ are not
smooth functions of $\x$. Nevertheless, by Weierstrass Preparation Theorem, in
a neighborhood of $\cv{\x}$, $Z_{k,1}+Z_{k,2}$ and $Z_{k,1}Z_{k,2}$ are smooth
functions of $\x$.  The gauss map $g$, by definition, has zeros (resp. poles)
at $Z_{k,1}$ and $Z_{k,2}$ for $k$ odd (resp. $k$ even).

\begin{proposition}
	For $(t,\x)$ close to $(0,\cv{\x})$, $\omega$ has two zeros in $\Omega_k$ for
	$k\in\Z$ (counting multiplicity) and no other zeros. The Regularity Problem
	is equivalent to
	\begin{equation} \label{eq:regularity1}
		\frac{\omega}{dz}(Z_{k,1})+\frac{\omega}{dz}(Z_{k,2})=0
	\end{equation}
	and
	\begin{equation} \label{eq:regularity2}
		\int_{\partial\Omega_k}g_k^{-1}\omega=0
	\end{equation}
	for $k\in\Z$.
\end{proposition}
\begin{proof}
	By \eqref{eq:omega0}, $\omega[0,\x]/dz$ is an elliptic function of degree $2$
	in $T_k$, with simple poles at $0_k$ and $v_k$.  So it has two zeros in
	$\Omega_k$ as long as $\varepsilon$ is small enough.  Hence
	$\psi^*\omega[0,\x]$ has two zeros in $\wt{\Omega}_k$, counting multiplicity.
	By Proposition \ref{prop:omega-smooth} and the Argument Principle, for
	$t$ small enough, $\psi^*\omega[t,\x]$ has two zeros in $\wt{\Omega}_k$, so
	$\omega[t,\x]$ has two zeros in $\psi(\wt{\Omega}_k)$.  By the same proof of
	Corollary 2 in \cite{traizet2013}, $\omega$ has no zero in the annuli
	$A_k^{\pm}$, so it has two zeros in $\Omega_k$ for each $k\in\Z$, and no
	further zero.

	If $Z_{k,1}\neq Z_{k,2}$, then $dz/g_k$ is a meromorphic 1-form on $T_k$ with
	two simple poles at $Z_{k,1}$ and $Z_{k,2}$.  Therefore, by the Residue
	Theorem,
	\[
		\frac{1}{g'_k(Z_{k,1})}+\frac{1}{g'_k(Z_{k,2})}=0,
	\]
	\begin{align*}
		\int_{\partial\Omega_k}g_k^{-1}\omega &= 2\pi\ii\left(\frac{\omega}{g'_k\,dz}(Z_{k,1})+\frac{\omega}{g'_k\,dz}(Z_{k,2})\right)\\
		&=\frac{2\pi\ii}{g'_k(Z_{k,1})}\left(\frac{\omega}{dz}(Z_{k,1})-\frac{\omega}{dz}(Z_{k,2})\right).
	\end{align*}
	Hence \eqref{eq:regularity1} and \eqref{eq:regularity2} are equivalent to
	$\frac{\omega}{dz}(Z_{k,1})=\frac{\omega}{dz}(Z_{k,2})=0$.

	If $Z_{k,1}=Z_{k,2}$, $g_k$ has a double zero at $Z_{k,1}$.  It is then easy
	to see, using the Residue Theorem, that \eqref{eq:regularity1} and
	\eqref{eq:regularity2} are equivalent to $\omega$ having a double zero at
	$Z_{k,1}$.
\end{proof}

\begin{proposition} \label{prop:regularity1}
	For $t$ in a neighborhood of $0$, there exists a unique value of
	$(b_k)\in\ellinf$, depending smoothly on $t$ and the other parameters, such
	that \eqref{eq:regularity1} is solved for all $k\in\Z$.  Moreover, at $t=0$,
	\[
		b_k=-a_k\xi(v_k;\tau_k)
	\]
	so
	\begin{equation} \label{eq:omega0-2}
		g_k\,dz=a_k\omega\quad\text{in $T_k$}.
	\end{equation}
\end{proposition}

\begin{proof}
	Define
	\[
		\cal{E}_k(t,\x)=\frac{\omega}{dz}(Z_{k,1})+\frac{\omega}{dz}(Z_{k,2}).
	\]
	Then by the Residue Theorem,
	\[
		\cal{E}_k(t,\x)=\frac{1}{2\pi\ii}\int_{\partial\Omega_k}\frac{\omega}{z-Z_{k,1}}+\frac{\omega}{z-Z_{k,2}}=\frac{1}{2\pi\ii}\int_{\partial\Omega_k}\frac{(2z-(Z_{k,1}+Z_{k,2}))\omega}{z^2-(Z_{k,1}+Z_{k,2})z+Z_{k,1}Z_{k,2}}.
	\]
	Using Proposition \ref{prop:omega-smooth}, $(\cal{E}_k)_{k\in\Z}$ is a smooth
	map with value in $\ellinf$. At $t=0$, we have by~\eqref{eq:omega0}
	\[
		\frac{\omega}{dz}=\frac{1}{a_k}(g_k(z)-b_k)-\xi(v_k;\tau_k)
	\]
	Since $g_k(Z_{k,i})=0$,
	\[
		\cal{E}_k(0,\x)=-2\left(\frac{b_k}{a_k}+\xi(v_k;\tau_k)\right).
	\]
	Since $\cv{a}_k=-1/2$, the partial differential of $(\cal{E}_k)$ with respect
	to $(b_k)$ is an automorphism of $\ellinf$. Proposition
	\ref{prop:regularity1} then follows from the Implicit Function Theorem.
\end{proof}

If $b_k$ is given by Proposition \ref{prop:regularity1}, then at $t=0$,
$\omega$ and $g_k$ are proportional in $T_k$ hence have the same zeros.  Then
\eqref{eq:regularity2} is satisfied at $t=0$ disregard of the value of the
other parameters.  So we cannot easily solve~\eqref{eq:regularity2} for $t\neq
0$ using the Implicit Function Theorem. We will solve~\eqref{eq:regularity2} in
Section \ref{ssec:balancing}.

\subsection{The Period Problem} \label{ssec:period-problem}

From now on, we assume that $(b_k)$ is given by
Proposition~\ref{prop:regularity1} and $\x=(a_k,v_k,\tau_k)_{k\in\Z}$ denotes
the remaining parameters.  The height differential has imaginary periods by
definition.  It remains to solve the following Period Problems for all
$k\in\Z$:
\begin{alignat}{2}
	\Re\int_{\alpha_k}\Phi_1&=(-1)^k,&\quad
	\Re\int_{\alpha_k}\Phi_2&=0, \label{period-problem-alpha}\\
	\Re\int_{\beta_k}\Phi_1&=\Re \tau,&
	\Re\int_{\beta_k}\Phi_2&=\Im \tau, \label{period-problem-beta}\\
	\Re\int_{\gamma}\Phi_1&=0,&
	\Re\int_{\gamma}\Phi_2&=0. \label{period-problem-gamma}
\end{alignat}

\begin{proposition} \label{prop:period-problem1}
	For $t$ small enough, there exists unique values for the parameters $(a_k)$
	and $(\tau_k)$ in $\ellinf$, depending smoothly on $t$ and $(v_k)$, such that
	\eqref{period-problem-alpha} and \eqref{period-problem-beta} are satisfied
	for all $k\in\Z$.  Moreover, at $t=0$, $a_k=-1/2$ and
	$\tau_k=(-\conj)^k(\tau)$, disregard of the value of $(v_k)$.
\end{proposition}

\begin{proof}
	we define
	\begin{align*}
		\cal{P}_{k,1}(t,\x)&=\conj\Big(\int_{\alpha_k} g^{-1}\,dh\Big) - \int_{\alpha_k} g\,dh\\
		\cal{P}_{k,2}(t,\x)&=\conj\Big(\int_{\beta_k} g^{-1}\,dh\Big) - \int_{\beta_k} g\,dh
	\end{align*}
	Equations \eqref{period-problem-alpha} and \eqref{period-problem-beta} are
	equivalent to
	\begin{equation} \label{period-problem2}
		\left\{\begin{array}{l}
				\cal{P}_{k,1}(t,\x)=2(-1)^k\\
		\cal{P}_{k,2}(t,\x)=2\tau.\end{array}\right.
	\end{equation}

	We have in $T_k$:
	\[
		g\,dh=\left\{\begin{array}{ll}
				g_k^{-1}\omega&\text{ if $k$ even}\\
				t^2 g_k\omega&\text{ if $k$ odd}
		\end{array}\right.
		\quad\text{and}\quad
		g^{-1}dh=\left\{\begin{array}{ll}
				t^2 g_k\omega&\text{ if $k$ even}\\
				g_k^{-1}\omega&\text{ if $k$ odd}
		\end{array}\right.
	\]
	We can take $\alpha_k=\psi_k(\wt{\alpha}_k)$ and
	$\beta_k=\psi_k(\wt{\beta}_k)$ where $\wt{\alpha}_k$ and $\wt{\beta}_k$ are
	fixed curves in $\wt{\Omega}_k$.  By Proposition~\ref{prop:omega-smooth},
	$(\cal{P}_{k,1})_{k\in\Z}$ and $(\cal{P}_{k,2})_{k\in\Z}$ are smooth maps
	with value in $\ellinf$.  At $t=0$, we have by~\eqref{eq:omega0-2} that: For
	$k$ even,
	\[
		\cal{P}_{k,1}(0,\x)=-\int_{\alpha_k}a_k^{-1}dz=\frac{-1}{a_k}
		\quad\text{and}\quad
		\cal{P}_{k,2}(0,\x)=-\int_{\beta_k}a_k^{-1}dz=\frac{-\tau_k}{a_k}.
	\]
	The solution to~\eqref{period-problem2} is then $a_k=-1/2$ and $\tau_k=\tau$.
	For $k$ odd,
	\[
		\cal{P}_{k,1}(0,\x)=\conj\int_{\alpha_k}a_k^{-1}dz=\conj\frac{1}{a_k}
		\quad\text{and}\quad
		\cal{P}_{k,2}(0,\x)=\conj\int_{\beta_k}a_k^{-1}dz=\conj\frac{\tau_k}{a_k}.
	\]
	The solution to~\eqref{period-problem2} is then $a_k=-1/2$ and
	$\tau_k=-\conj{\tau}$.  The partial differential of
	$((\cal{P}_{k,1}),(\cal{P}_{k,2}))$ with respect to $((a_k),(\tau_k))$ is
	clearly an automorphism of $\ellinf\times\ellinf$.
	Proposition~\ref{prop:period-problem1} then follows from the Implicit
	Function Theorem.
\end{proof}

\subsection{Balancing} \label{ssec:balancing}

From now on, we assume that the parameters $(a_k)$ and $(\tau_k)$ are given by
Proposition \ref{prop:period-problem1}.  So the only remaining parameters are
$t$ and $\bm v=(v_k)$.  It remains to solve \eqref{eq:regularity2} and
\eqref{period-problem-gamma}.  We define
\[
	\cal{G}_k(t,\bm v)=\conj^k\int_{\partial D_k^-}g_k\omega.
\]

\begin{proposition} \label{prop:balancing1}
	For $t\neq 0$, \eqref{eq:regularity2} and \eqref{period-problem-gamma} are
	equivalent to $\cal{G}_k(t,\bm v)=\cal{G}_0(t,\bm v)$ for all $k\in\Z$.
\end{proposition}

\begin{proof}
	We have for $k\in\Z$:
	\begin{align*}
		\int_{\partial\Omega_k}g_k^{-1}\omega
		&= -\int_{\partial D_k^-}z_k^-\omega-\int_{\partial D_k^+}z_k^+\omega\\
		&=\int_{\partial D_{k-1}^+}\frac{t^2}{z_{k-1}^+}\omega
		+\int_{\partial D_{k+1}^-}\frac{t^2}{z_{k+1}^-}\omega\\
		&=t^2\int_{\partial D_{k-1}^+}g_{k-1}\omega+t^2\int_{\partial D_{k+1}^-}g_{k+1}\omega\\
		&=-t^2\int_{\partial D_{k-1}^-}g_{k-1}\omega+t^2\int_{\partial D_{k+1}^-}g_{k+1}\omega
		\quad\text{(because $g_{k-1}\omega$ is holomorphic in $\Omega_{k-1}$)}\\
		&=t^2\conj^{k+1}(\cal{G}_{k+1}-\cal{G}_{k-1}).
	\end{align*}
	Hence \eqref{eq:regularity2} is equivalent to
	$\cal{G}_{k+1}=\cal{G}_{k-1}$ for $k\in\Z$.

	We chose $\partial D_1^-$ as a representative of $\gamma$. Equation
	\eqref{period-problem-gamma} is equivalent to
	\[
		\int_{\partial D_1^-}g^{-1}\,dh-\conj\Big(\int_{\partial D_1^-}g\,dh\Big)=0.
	\]
	We have
	\[
		\int_{\partial D_1^-}\overline{g\,dh}=t^2\int_{\partial D_1^-}\overline{g_1\omega}=t^2\cal{G}_1.
	\]
	\[
		\int_{\partial D_1^-}g^{-1}\,dh=\int_{\partial D_1^-}g_1^{-1}\omega
		=-t^2\int_{\partial D_0^+}g_0\omega
		=t^2\int_{\partial D_0^-}g_0\omega
		=t^2\cal{G}_0.
	\]
	Hence \eqref{period-problem-gamma} is equivalent to $\cal{G}_1=\cal{G}_0$.
\end{proof}

\begin{proposition} \label{prop:balancing2}
	Assume that the configuration $\bm{q}=(q_k)$ is balanced and non-degenerate.
	For $t$ in a neighborhood of $0$, there exists a unique $\bm v(t)$ in
	$\ellinf$, depending smoothly on $t$, such that $v_k(0)=(-\conj)^k q_k$ and
	\begin{equation} \label{eq:normalisation}
		\cal{G}_k(t,\bm v(t))=-2\pi\ii G(q_0;\tau)
	\end{equation}
	for all $k\in\Z$, so \eqref{eq:regularity2} and \eqref{period-problem-gamma}
	are solved.
\end{proposition}

\begin{proof}
	First of all, $(\cal{G}_k(t,\bm v))$ is a smooth function of $(t,\bm v)$ with
	value in $\ellinf$ by Proposition~\ref{prop:omega-smooth}.  At $t=0$, we have
	in $T_k$:
	\[
		g_k\frac{\omega}{dz}=a_k\left(\frac{\omega}{dz}\right)^2=\frac{-1}{2}\left(\zeta(z;\tau_k)-\zeta(z-v_k;\tau_k)-\xi(v_k;\tau_k)\right)^2
	\]
	\begin{align*}
		\Res_{0_k}(g_k\omega)&=\frac{-1}{2}\Res_{0_k}\left[
		\zeta(z;\tau_k)^2-2\zeta(z;\tau_k)\big(\zeta(z-v_k;\tau_k)+\xi(v_k;\tau_k)\big)\right]\\
		&=-\zeta(v_k;\tau_k)+\xi(v_k;\tau_k)=-G(v_k;\tau_k).
	\end{align*}
	Here we used the fact that $\zeta$ is odd, hence $\zeta^2$ is even and has no
	residue at $0$.  Then by the Residue Theorem,
	\[
		\cal{G}_k(0,\bm v)=\conj^k\big(-2\pi\ii\,G(v_k;\tau_k)\big)
		=-2\pi\ii(-\conj)^k G(v_k;\tau_k).
	\]
	Recall that $\tau_k=(-\conj)^k(\tau)$ at $t=0$.  Now change to the variable
	$u_k = (-\conj)^k v_k$ with central value $\cv u_k = (-\conj)^k \cv v_k =
	q_k$.  Using the definition of $\zeta$, $\xi$ and $G$, one easily checks that
	\[
		G(-\overline{z};-\overline{\tau})=-\conj{G(z;\tau)}.
	\]
	Hence
	\[
		\cal{G}_k(0,\bm v)=-2\pi\ii \,G(u_k;\tau).
	\]
	Then Proposition \ref{prop:balancing2} follows from the balance and
	non-degeneracy of $(q_k)$ and the Implicit Function Theorem.
\end{proof}

\begin{remark}
	The horizontal component of the flux of $\gamma$, identified with a complex
	number, is given by
	\[
		-\ii\int_{\gamma} g\,dh=-\ii\int_{\partial D_1^-}t^2 g_1\omega=-\ii t^2\overline{\cal{G}_1}.
	\]
	So \eqref{eq:normalisation} for $k=1$ normalizes the horizontal part of the
	flux of $\gamma$ (which can also be taken as a free parameter).
\end{remark}

\subsection{Embeddedness} \label{ssec:embedded}

We denote by $\x(t)$ the value of the parameters given by Proposition
\ref{prop:regularity1}, \ref{prop:period-problem1} and \ref{prop:balancing2},
$(\Sigma_t,g_t,dh_t)$ the corresponding Weierstrass data,
$f_t:\Sigma_t\to\R^3/\Gamma$ the immersion given by Weierstrass Representation
and $M_t=f_t(\Sigma_t)$.  Recall that $\Gamma$ is the 2-dimensional lattice
generated by the horizontal vectors $(1,0,0)$ and $(\Re\tau, \Im\tau, 0)$.  The
goal of this section is to prove that $M_t$ is embedded and has the geometry
described in Section \ref{hyp:separate}.  The argument is very similar to
Section 4.10 of \cite{morabito2012}, so we will only sketch it.

\medskip

We write $f_t=(X_t,h_t)$ with
\[
	X_t(z)=\frac{1}{2}\conj\Big(\int_{z_0}^z g_t^{-1}dh_t\Big)-\frac{1}{2}\int_{z_0}^z g_t\,dh_t
	\quad\text{and}\quad h_t(z)=\Re\int_{z_0}^z dh_t.
\]
We fix a base point $\wt{O}_k$ in $\wt{\Omega}_k$, away from the zeros of
$g_k\circ\psi_k$, and let $O_k=\psi_k(\wt{O}_k)\in\Omega_k$.  Let
$w_k(t)\in\Sigma_t$ be the point $z_k^+=t$ which is identified with the point
$z_{k+1}^-=t$ (the ``middle'' of the $k$-th neck).  For $r>0$, we denote
$\Omega_{k,r}$ the torus $T_k$ minus the two disks $|z_k^{\pm}|\leq r$.  By the
computations in Section \ref{ssec:period-problem}, we have in $\Omega_{k,r}$
\begin{equation} \label{limit-dXt}
	\lim_{t\to 0}dX_t=(-\conj)^k dz
\end{equation}
so the image of $\Omega_{k,r}$ is a graph for $t$ small enough and
\[
	\lim_{t\to 0}X_t(w_k(t))-X_t(w_{k-1}(t))=(-\conj)^k\int_{0_k}^{\cv{v}_k}dz
	=(-\conj)^k(\cv{v}_k)=q_k.
\]
\begin{remark}
	This computation is not rigorous because the limit \eqref{limit-dXt} only
	holds in $\Omega_{k,r}$. It can be made rigorous by expanding $\omega_t$ in
	Laurent series in the annuli $A^\pm_k$, see details in Appendix A of
	\cite{morabito2012}.
\end{remark}
Recall that $q_k=p_k-p_{k-1}$.  We may translate $f_t$ horizontally so that
$X_t(w_0(t))=p_0$. Then
\[
	\lim_{t\to 0}X_t(w_k(t))=p_k.
\]
In other words, $p_k$ is the limit position of the $k$-th neck.  Note however
that the convergence is not uniform with respect to $k$. By~\eqref{eq:omega0},
we have for $z\in\Omega_{k,r}$
\begin{multline*}
	\lim_{t\to 0}\frac{1}{t}(h_t(z)-h_t(O_k))= \Re\,\lim_{t\to 0}\int_{O_k}^z
	\omega_t\\
	= \Re\int_{O_k}^z\left(\zeta(z;\cv{\tau}_k)-\zeta(z-\cv{v}_k;\cv{\tau}_k)-\xi(\cv{v}_k;\cv{\tau}_k)\right)\,dz=:\Upsilon_k(z).
\end{multline*}
The function $\Upsilon_k$ is bounded in $\Omega_{k,r}$ by a uniform constant $C(r)$
depending only on $r$.  By Lemma A.2 in \cite{morabito2012},
\[
	\int_{O_{k-1}}^{O_k}\omega_t\simeq -2\log t\Res_{0_k}(\omega_0)=-2\log t\quad\text{as $t\to 0$.}
\]
Hence
\[
	h_t(O_k)-h_t(O_{k-1}) \sim -2t\log t.
\]
This ensures that for $t$ small enough, $f_t(\Omega_{k,r})$ lies strictly above
$f_t(\Omega_{k-1,r})$ so the images $f_t(\Omega_{k,r})$ for $k\in\Z$ are
disjoint.

Since the function $\Upsilon_k$ has two logarithmic singularities at $0_k$ and
$\cv{v}_k$, for $c$ large enough, the level lines $\Upsilon_k=\pm c$ are convex
curves, which are included in $\Omega_{k,r}$ provided $r$ is small enough.
Define
\[
	h_k^{\pm}=h_t(O_k)\pm tc.
\]
Then for $t$ small enough, $M_t$ intersects the planes $x_3=h_k^+$ and
$x_3=h_k^-$ in two convex curves denoted $\gamma_k^+$ and $\gamma_k^-$, which
are included in $f_t(\Omega_{k,r})$. Then $M_t\cap\{h_k^+<x_3<h_{k+1}^-\}$ is a
minimal annulus bounded by two convex curves in parallel planes.  Such an
annulus is foliated by convex curves by a theorem of Shiffman
\cite{shiffman1956}.  This proves that $M_t$ is embedded.

\section{Convergence to TPMSs}\label{sec:asymptotic-immersion}

In this section, we study the asymptotic behavior of the minimal surfaces that
we just constructed.  Readers who are not interested in the asymptotic behavior
may skip this part and jump directly to the next section.

Assume that the configuration $(q_k)$ is periodic. Then the corresponding
minimal surface $M_t$ is a TPMS (as an easy consequence of uniqueness in the
Implicit Function Theorem).  Let $(q_k')$ be another balanced, non-degenerate
configuration with the same horizontal lattice $\Gamma$, and assume that
$q'_k=q_k$ for all $k\geq 0$.  Let $M'_t$ be the family of minimal surfaces
corresponding to the configuration $(q'_k)$.  In this section, we prove that
$M'_t$ is asymptotic to a translation of $M_t$ as the vertical coordinate $x_3
\to +\infty$.  This finally justify the term ``TPMS twinning''.

We use primes for all objects associated to the configuration $(q'_k)$.  So
$\Sigma'=\Sigma[t,\x'(t)]$ and $f':\Sigma'\to\R^3/\Gamma$ is the immersion
given by the Weierstrass data $g'=g[t,\x'(t)]$ and $dh'=t\omega[t,\x'(t)]$.  We
will omit the dependence on $t$, hence will write, for instance, $\omega =
\omega[t,\x(t)]$ and $\omega'=\omega[t,\x'(t)]$, and in the same way
$\psi_k=\psi_k[t,\x(t)]$ and $\psi'_k=\psi_k[t,\x'(t)]$.  Otherwise notations
are as in Section~\ref{ssec:embedded}.

We use the same letter $C$ to denote any constant that is independent of $t>0$
and $k\in\Z$.

\medskip

Fix an arbitrary $\delta > 1$.  We will prove in
Proposition~\ref{prop:parameter-decay} that, for $t$ sufficiently small,
$\x-\x'$ decays like $\delta^{-k}$.  That is
\[
	\|x_k - x'_k\| \leq \frac{C}{\delta^k}.
\]
Then Proposition~\ref{prop:omega-decay} implies that, for $t$ sufficiently
small, the difference between $\psi_k^*\omega$ and $(\psi'_k)^*\omega'$ in
$\wt\Omega_k$ also decays like $\delta^{-k}$; see
Equation~\eqref{eq-omega-decay} below.

It is convenient to scale the third coordinate of $f$ and $f'$ by
$t^{-1}$.  So let $S:\R^3\to\R^3$ be the linear map defined by
$S(x_1,x_2,x_3)=(x_1,x_2, t^{-1}x_3)$ and define $\wh{f}=S\circ f$ and
$\wh{f}'=S\circ f'$.

\begin{proposition} \label{prop:immersion-decay1} For $k\in\N$,
	\[
		\| d(\wh{f}\circ\psi_k) - d(\wh{f}'\circ\psi_k') \|_{C^0(\wt{\Omega}_k)} \leq \frac{C}{\delta^k}
	\]
\end{proposition}

\begin{proof}
	We have in $\Omega_k$
	\begin{equation} \label{eq-df}
		d\wh{f}=\left\{\begin{array}{l}
				(\frac{1}{2}(t^2\overline{g_k\omega}-g_k^{-1}\omega),\Re\omega)
				\quad\text{if $k$ is even}\\
				(\frac{1}{2}(\overline{g_k^{-1}\omega}-t^2g_k\omega),\Re\omega)
		\quad\text{if $k$ is odd}\end{array}\right.
	\end{equation}
	and similar formulas for $d\wh{f}'$.  By Propositions \ref{prop:omega-decay}
	and \ref{prop:parameter-decay}, we have
	\begin{equation} \label{eq-omega-decay}
		\|\psi_k^*\omega- (\psi'_k)^*\omega'\|_{C^0(\wt{\Omega}_k)}\leq\frac{C}{\delta^k}.
	\end{equation}
	Using Proposition \ref{prop:parameter-decay} and the definition of $g_k$, we
	obtain
	\begin{equation} \label{eq-gkomega-decay}
		\|\psi_k^*(g_k\omega)- (\psi'_k)^*(g_k'\omega')\|_{C^0(\wt{\Omega}_k)}\leq\frac{C}{\delta^k}.
	\end{equation}
	Let $\wt{D}_{k,1}$ and $\wt{D}_{k,2}$ be two disks containing the two zeros
	of $g_k\circ\psi_k$, so that $g_k^{-1}\circ\psi_k$ is uniformly bounded
	outside these disks. Then
	\begin{equation} \label{eq-gkinvomega-decay}
		\|\psi_k^*(g_k^{-1}\omega)- (\psi'_k)^*((g_k')^{-1}\omega')\|_{C^0(\wt{\Omega}_k\setminus\wt{D}_{k,1}\cup\wt{D}_{k,2})}\leq\frac{C}{\delta^k}.
	\end{equation}
	Since the Regularity Problem is solved, $(g'_k)^{-1}\omega'$ extends
	holomorphically to the zeros of $g'_k$.  Since $\psi_k(x+\ii y) -
	\psi_k'(x+\ii y) = (\tau_k-\tau_k')y$, we have
	\begin{equation} \label{eq-gkinvomega-decay2}
		\|\psi_k^*((g'_k)^{-1}\omega')- (\psi'_k)^*((g_k')^{-1}\omega')\|_{C^0(\wt{\Omega}_k)}\leq\frac{C}{\delta^k}.
	\end{equation}
	From \eqref{eq-gkinvomega-decay} and \eqref{eq-gkinvomega-decay2}, we obtain
	by the Triangular Inequality
	\[
		\|\psi_k^*\left[g_k^{-1}\omega-(g_k')^{-1}\omega'\right]\|_{C^0(\partial \wt{D}_{k,i})}\leq\frac{C}{\delta^k}.
	\]
	Since $g_k^{-1}\omega-(g'_k)^{-1}\omega'$ is holomorphic, we obtain by the
	Maximum Principle (for the holomorphic structure on $\T$ induced by $\psi_k$)
	\[
		\|\psi_k^*\left[g_k^{-1}\omega-(g_k')^{-1}\omega'\right]\|_{C^0(\wt{D}_{k,i})}\leq\frac{C}{\delta^k}.
	\]
	Using \eqref{eq-gkinvomega-decay2},
	\begin{equation} \label{eq-gkinvomega-decay3}
		\|\psi_k^*(g_k^{-1}\omega)- (\psi'_k)^*((g_k')^{-1}\omega')\|_{C^0(\wt{D}_{k,i})}\leq\frac{C}{\delta^k}.
	\end{equation}
	Proposition \ref{prop:immersion-decay1} follows from \eqref{eq-omega-decay},
	\eqref{eq-gkomega-decay}, \eqref{eq-gkinvomega-decay} and
	\eqref{eq-gkinvomega-decay3}.
\end{proof}

Recall that $O_k=\psi_k(\wt{O}_k)$ and $O'_k=\psi'_k(\wt{O}_k)$.
\begin{proposition}
	\label{prop:immersion-decay2}
	For $k\in\N$,
	\[
		\|\wh{f}(O_{k+1})-\wh{f}(O_k)-\wh{f}'(O'_{k+1})+\wh{f}'(O'_k)\|\leq \frac{C}{\delta^k}.
	\]
	Hence the sequence $(\wh{f}(O_k)-\wh{f}'(O_k'))_{k\in\N}$ is Cauchy. By
	translation, we may assume that its limit is zero, so
	\[
		\|\wh{f}(O_k)-\wh{f}'(O_k')\|\leq\frac{C'}{\delta^k}\quad\text{ with }C'=\frac{C\delta}{\delta-1}.
	\]
\end{proposition}

\begin{proof}
	Recall from Section \ref{ssec:embedded} that $w_k\in\Sigma$ denotes the
	point $z_k^+=t$, identified with $z_{k+1}^-=t$.  Similarly, we introduce
	$w'_k\in\Sigma'$ to denote the point $z'^+_k=t$, identified with
	$z'^-_{k+1}=t$.  Moreover, let $m_k^{\pm}\in\Sigma$ denotes the point
	$z_k^{\pm}=\varepsilon$, and $m'^\pm_k\in\Sigma'$ denotes the point
	$z'^\pm_k=\varepsilon$.

	Proposition \ref{prop:immersion-decay2} follows from the following three
	estimates: For $k\in\N$:
	\begin{align}
		\|\wh{f}(O_k)-\wh{f}(m_k^{\pm})-\wh{f}'(O'_k)+\wh{f}'(m'^\pm_k)\| &\leq \frac{C}{\delta^k}, \label{eq-neck1}\\
		\|\wh{f}(m_k^+)-\wh{f}(w_k)-\wh{f}'(m'^+_k)+\wh{f}'(w'_k)\| &\leq\frac{C}{\delta^k}, \label{eq-neck2}\\
		\|\wh{f}(m_{k+1}^-)-\wh{f}(w_k)-\wh{f}'(m'^-_{k+1})+\wh{f}'(w'_k)\| &\leq\frac{C}{\delta^k}. \label{eq-neck3}
	\end{align}

	Inequality \eqref{eq-neck1} follows from Proposition \ref{prop:immersion-decay1},
	\[
		\psi_k^{-1}(O_k)=(\psi'_k)^{-1}(O'_k)=\wt{O}_k
		\quad\text{and}\quad
		|\psi_k^{-1}(m_k^{\pm})-(\psi'_k)^{-1}(m'^\pm_k)|\leq\frac{C}{\delta_k}.
	\]
	To prove \eqref{eq-neck2}, we follow the proof of Lemma A.1 in
	\cite{morabito2012}.  We write the Laurent series of $\omega$ in the annulus
	$A^+_k$ in term of the complex coordinate $z=z_k^+$ as
	\[
		\omega=\frac{-dz}{z}+\sum_{n\geq 1}c_{k,n}^+z^{n-1}dz
		+\sum_{n\geq 1} t^{2n} c_{k,n}^-\frac{dz}{z^{n+1}}
	\]
	where
	\begin{gather*}
		c_{k,n}^+=\frac{1}{2\pi\ii}\int_{|z|=\varepsilon}\frac{\omega}{z^n}
		=\frac{1}{2\pi\ii}\int_{\partial D_k^+}\frac{\omega}{(z_k^+)^n}\\
		c_{k,n}^-=\frac{t^{-2n}}{2\pi\ii}\int_{|z|=\varepsilon}z^n\omega
		=\frac{1}{2\pi\ii}\int_{\partial D_k^+}(t^{-2}z_k^+)^n\omega
		=\frac{-1}{2\pi\ii}\int_{\partial D_{k+1}^-}\frac{\omega}{(z_{k+1}^-)^n}.
	\end{gather*}
	We expand $\omega'$ in the annulus $A'^+_k$ in the same way, with
	coefficients $c'^\pm_{k,n}$ given by similar formulas.  Using
	estimates~\eqref{eq-omega-decay} and~\eqref{eq-gkomega-decay}, we obtain the
	following estimate
	\[
		|c_{k,n}^{\pm}-c'^\pm_{k,n}|\leq \frac{Cn}{\delta^k(2\varepsilon)^n}.
	\]
	By integration, we obtain the following estimates (see details in Appendix A of \cite{morabito2012}):
	\begin{align*}
		\left|\int_{z=\varepsilon}^t \omega-\omega'\right| &\leq \frac{C}{\delta^k}\\
		\left|\int_{z=\varepsilon}^t z(\omega-\omega')\right| &\leq \frac{C}{\delta^k}\\
		\left|\int_{z=\varepsilon}^t t^2z^{-1}(\omega-\omega')\right| &\leq \frac{C}{\delta^k}
	\end{align*}
	Then \eqref{eq-neck2} follows from \eqref{eq-df}.  \eqref{eq-neck3} is proved
	in the same way, using $z=z_{k+1}^-$ as a local coordinate.
\end{proof}

\def\theperiod{{\cal T}}

Assume that the configuration $(q_k)$ is periodic with even period $N$, i.e.\
$q_{k+N}=q_k$. By uniqueness in the Implicit Function Theorem, the resulting
immersion $f$ is periodic. More precisely, if we define
$\sigma:\Sigma\to\Sigma$ by $z\in T_k\mapsto z\in T_{k+N}$ then
$f\circ\sigma=f+\theperiod$ (where the period $\theperiod\in\R^3$ depends on
$t$).

\begin{proposition} \label{prop:TPMS-decay}
	Let $M=f(\Sigma)$ and $M'=f'(\Sigma')$.  Define $M'_{\ell}=M'-\ell\theperiod$ for
	$\ell\in\N$. Then
	\[
		\lim_{\ell\to\infty} M'_{\ell}=M.
	\]
	Here the limit is for the smooth convergence on compact subsets of
	$\R^3/\Gamma$.
\end{proposition}

\begin{proof}
	By periodicity we have
	\[
		\wt{\Omega}_{k+N}=\wt{\Omega}_k\quad\text{and}\quad
		f\circ\psi_{k+N}=f\circ\psi_k+\theperiod\quad\text{in $\wt{\Omega}_k$}.
	\]
	By Propositions \ref{prop:immersion-decay1} and \ref{prop:immersion-decay2},
	we have for $\ell\in\N$
	\[
		\|f\circ\psi_{k+\ell N}-f'\circ\psi'_{k+\ell N}\|_{C^0(\wt{\Omega}_k)}\leq\frac{C}{\delta^{k+\ell N}}.
	\]
	Define $f'_{\ell}=f-\ell\theperiod$. Then
	\[
		\|f\circ\psi_k-f'_{\ell}\circ\psi'_{k+\ell N}\|_{C^0(\wt{\Omega}_k)}\leq\frac{C}{\delta^{k+\ell N}}.
	\]
	Hence
	\begin{equation}
		\label{eq-limit-fl}
		\lim_{\ell\to\infty} f'_{\ell}\circ\psi'_{k+\ell N}(\wt{\Omega}_k)=f\circ\psi_k(\wt{\Omega}_k).
	\end{equation}
	Recall from Section \ref{ssec:embedded} that we have defined heights
	$h_k^{\pm}$ such that $M\cap\{h_k^-<x_3<h_k^+\}$ is included in
	$f\circ\psi_k(\wt{\Omega}_k)$ and is bounded by two convex curves denoted
	$\gamma_k^+$ and $\gamma_k^-$.  Then for $\ell$ large enough,
	$M'_{\ell}\cap\{h_k^-<x_3<h_k^+\}$ is included in
	$f'_{\ell}\circ\psi'_{k+\ell N}(\wt{\Omega}_k)$ and is bounded by two convex
	curves denoted $\gamma'^+_{k,\ell}$ and $\gamma'^-_{k,\ell}$.  By
	\eqref{eq-limit-fl} we have
	\[
		\lim_{\ell\to\infty} M'_{\ell}\cap\{h_k^-<x_3<h_k^+\}=M\cap\{h_k^-<x_3<h_k^+\}.
	\]
	Let $A'_{k,\ell}=M'_{\ell}\cap\{h_k^+<x_3<h_{k+1}^-\}$.  Then $A'_{k,\ell}$
	is an unstable minimal annulus bounded by $\gamma'^+_{k,\ell}$ and
	$\gamma'^-_{k+1,\ell}$.  By Theorem 2(a) in \cite{traizet2010}, $A'_{k,\ell}$
	converges subsequentially as $\ell\to\infty$ to an unstable annulus bounded
	by $\gamma_k^+$ and $\gamma_{k+1}^-$. By \cite{meeks-white}, this annulus is
	unique so the whole sequence converges and
	\[
		\lim_{\ell\to\infty} A'_{k,\ell}=M\cap\{h_k^+<x_3<h_{k+1}^-\}.
	\]
	Hence for any integers $k_1<k_2$, we have
	\[
		\lim_{\ell\to\infty} M'_{\ell}\cap\{h_{k_1}^+<x_3<h_{k_2}^+\}=M\cap\{h_{k_1}^+<x_3<h_{k_2}^+\}
	\]
	which proves Proposition \ref{prop:TPMS-decay} since $\lim_{k\to\pm\infty}h_k^+=\pm\infty$.
\end{proof}

\section{The holomorphic 1-forms \texorpdfstring{$\omega$}{[omega]}} \label{sec:omega}

The goal of this section is to prove:
\begin{proposition} \label{prop:omega}
	\leavevmode
	\begin{enumerate}
		\item For $(t,\x)$ in a neighborhood of $(0,\cv{\x})$, there exists a
			holomorphic (regular if $t=0$) 1-form $\omega[t,\x]$ on $\Sigma[t,\x]$
			with imaginary periods on $\alpha_k$ and $\beta_k$ for all $k\in\Z$ and
			$\int_{\gamma}\omega=2\pi\ii$.

		\item At $t=0$, we have for all $k\in\Z$:
			\[
				\omega[0,\x]=\left(\zeta(z;\tau_k)-\zeta(z-v_k;\tau_k)-\xi(v_k;\tau_k)\right)dz\quad\text{in $T_k$}.
			\]

		\item The pullback $\psi^*\omega$ is in $C^0(\wt{\Omega})$ and depends
			smoothly on $(t,\x)$ in a neighborhood of $(0,\cv{\x})$.
	\end{enumerate}
\end{proposition}

\begin{remark}
	If the configuration is periodic with period $2N$ (namely, $q_{k+2N}=q_k$),
	the quotient of $\Sigma$ by its period is a compact Riemann surface, obtained
	by opening $2N$ nodes between $2N$ tori.  In this case, the existence of
	$\omega$ follows from the standard theory of Opening Nodes \cite{fay1973}. To
	prove the existence of $\omega$ in the non-periodic case, we adapt the
	argument in \cite{traizet2013} which allows for infinitely many nodes. The
	difference is that Riemann spheres are replaced by tori.
\end{remark}

In the following, we use the same letter $C$ to designate all uniform
constants.

\subsection{Preliminaries}

\begin{definition} \label{def-differentials}
	\leavevmode
	\begin{enumerate}
		\item For $p,q\in T_k$, $p\neq q$, we denote by $\omega_{k,p,q}$ the unique
			meromorphic 1-form on $T_k$ with simple poles at $p,q$ with residues $1$
			and $-1$, and imaginary periods on $\alpha_k$ and $\beta_k$.  So
			$\omega_{k,p,q}$ is an abelian differential of the third kind with real
			normalisation.

		\item For $n\geq 2$, we denote $\omega_{k,n}^+$ (resp. $\omega_{k,n}^-$)
			the unique meromorphic 1-form on $T_k$ with a pole of multiplicity $n$ at
			$v_k$ (resp. $0_k$) with principal part
			\[
				\frac{dz_k^{\pm}}{(z_k^{\pm})^n}
			\]
			and imaginary periods on $\alpha_k$ and $\beta_k$.  So $\omega_{k,n}^\pm$
			are abelian differentials of the second kind with real normalisation.
			Recall that it depends on the choice of the local coordinate $z_k^{\pm}$
			used to define the principal part.
	\end{enumerate}
\end{definition}

\begin{lemma} \label{omega-lemma1}
	The abelian differential $\omega_{k,p,q}$ is explicitely given by
	\begin{equation} \label{eq-etakpq}
		\omega_{k,p,q}=\left(\zeta(z-p;\tau_k)-\zeta(z-q;\tau_k)-\xi(q-p;\tau_k)\right)dz.
	\end{equation}
\end{lemma}

\begin{proof}
	Let $\omega'_{k,p,q}$ be the right-hand side of \eqref{eq-etakpq}.  Then
	$\omega'_{k,p,q}$ has simple poles at $p,q$ with residues $1$ and $-1$.
	Using the quasi-periodicity of $\zeta$, $\omega'_{k,p,q}$ is independent of
	the choice of the representatives of $p$ and $q$ modulo $\Z+\tau_k\Z$.  We
	take these representatives in the fundamental parallelogram spanned by $1$
	and $\tau_k$.  We may represent $\alpha_k$ and $\beta_k$ by curves which do
	not intersect the segment $[p,q]$.  Using the quasi-periodicity of $\zeta$
	(and omitting the argument $\tau_k$ everywhere)
	\[
		\frac{\partial}{\partial q}\int_{\alpha_k}(\zeta(z-p)-\zeta(z-q))dz=
		\int_{\alpha_k}\frac{d\zeta}{dz}(z-q)dz=\left[\zeta(z-q)\right]_{\alpha_k(0)}^{\alpha_k(1)}=\eta_1.
	\]
	Hence
	\[
		\int_{\alpha_k}(\zeta(z-p)-\zeta(z-q))dz=(q-p)\eta_1.
	\]
	Write $q-p=x + y \tau_k$ and recall the definition of $\xi$ and Legendre
	Relation $\eta_1\tau_k-\eta_2=2\pi\ii$, where we assumed that $\Im\tau > 0$.
	Then
	\[
		\int_{\alpha_k}\omega'_{k,p,q}=(x+y\tau_k)\eta_1-(x\eta_1+y\eta_2)=y(\eta_1\tau_k-\eta_2)=2\pi\ii y.
	\]
	In the same way,
	\[
		\int_{\beta_k}\omega'_{k,p,q}=(x+y\tau_k)\eta_2-(x\eta_1+y\eta_2)\tau_k=x(\eta_2-\eta_1\tau_k)=-2\pi\ii x.
	\]
	Hence $\omega'_{k,p,q}$ has imaginary periods so is equal to $\omega_{k,p,q}$
	by uniqueness.
\end{proof}

\begin{lemma} \label{omega-lemma2}
	There exists a uniform constant $C$ such that for $(t,\x)$ in a neighborhood of
	$(0,\cv{\x})$ and $k\in\Z$:
	\[
		\|\omega_{k,n}^{\pm}\|_{C^0(\Omega_k)}\leq C\left(\frac{2}{\varepsilon}\right)^{n-1}.
	\]
\end{lemma}

\begin{proof}[Proof]
	We only prove the $+$ case.  The $-$ case follows similarly.

	Let $\eta_{k,n}^+$ be the unique meromorphic 1-form on $T_k$ with a pole at
	$v_k$ with the same principal part as $\omega_{k,n}^+$ and normalized by
	$\int_{\alpha_k}\eta_{k,n}^+=0$. These two differentials are related by
	\begin{equation}
		\label{omega-eta}
		\omega_{k,n}^+=\eta_{k,n}^+ +\frac{\ii}{\Im(\tau_k)}\left(\Re\int_{\beta_k}\eta_{k,n}^+\right)dz.
	\end{equation}
	Indeed, the right-hand side has imaginary periods.

	Write $\eta_{k,n}^+=f_{k,n}^+dz.$ For $p,q$ in $T_k\setminus D_k^+$, we have
	by the Residue Theorem in $T_k\setminus D_k^+$:
	\[
		\int_{\partial D_k^+}f_{k,n}^+\omega_{k,p,q}=-2\pi\ii\left(f_{k,n}^+(p)-f_{k,n}^+(q)\right).
	\]
	On the other hand, since $\omega_{k,p,q}$ is holomorphic in $D_k^+$ and by
	definition of the principal part of $f_{k,n}^+$ at $v_k$:
	\[
		\int_{\partial D_k^+} f_{k,n}^+\omega_{k,p,q}
		=\int_{\partial D_k^+}\frac{dz_k^+}{dz}(z_k^+)^{-n}\omega_{k,p,q}.
	\]
	Hence
	\[
		f_{k,n}^+(p)-f_{k,n}^+(q)=-\chi_{k,n}^+(p,q)\quad\text{with}\quad
		\chi_{k,n}^{+}(p,q)=\frac{1}{2\pi\ii}\int_{\partial D_k^{+}}\frac{dz_k^{+}}{dz}(z_k^{+})^{-n}\omega_{k,p,q}.
	\]
	Integrating with respect to $q$ on $\alpha_k$, we obtain the following
	integral representation of $f_{k,n}^+(p)$ for $p\in T_k\setminus D_k^+$:
	\begin{equation}
		\label{integral-representation}
		f_{k,n}^+(p)=\int_{\alpha_k}(f_{k,n}^+(p)-f_{k,n}^+(q))dq=-\int_{\alpha_k}\chi_{k,n}^+(p,q)dq.
	\end{equation}
	By Cauchy Theorem, we may replace the circle $\partial D_k^+$ by the circle
	$|z_k^+|=\frac{\varepsilon}{2}$ in the definition of $\chi_{k,n}^+$.  Using
	Lemma \ref{omega-lemma1}, there exists a uniform constant $C$ (independent of
	$k\in\Z$ and $\x$ in a neighborhood of $\cv{\x}$) such that for all $p,q\in
	\Omega_k$ and $z$ on the circle $|z_k^+|=\frac{\varepsilon}{2}$:
	\[
		|\omega_{k,p,q}(z)|\leq C.
	\]
	Then for $p,q\in\Omega_k$:
	\[
		|\chi_{k,n}^+(p,q)|\leq \frac{C}{2\pi}\int_{|z_k^+|=\varepsilon/2}\frac{|dz_k^+|}{|z_k^+|^n}
		=\frac{C}{2\pi}\left(\frac{2}{\varepsilon}\right)^n 2\pi\frac{\varepsilon}{2}
		=C\left(\frac{2}{\varepsilon}\right)^{n-1}.
	\]
	By \eqref{integral-representation},
	\begin{equation}
		\label{estimate-eta}
		\|\eta_{k,n}^+\|_{C^0(\Omega_k)}\leq C\left(\frac{2}{\varepsilon}\right)^{n-1}.
	\end{equation}
	Lemma \ref{omega-lemma2} then follows from \eqref{omega-eta}.
\end{proof}

\subsection{Existence}

Let $\Lambda$ denotes the set $\Z\times\{n\in\N:n\geq 2\}\times\{+,-\}$.  We
look for $\omega$ in the form $\omega[t,\x]=\wt{\omega}(\x,\bm\lambda(t,\x))$
where
\begin{equation} \label{eq-omega-series}
	\wt{\omega}(\x,\bm\lambda)=\omega_{k,0_k,v_k}+\sum_{n=2}^{\infty}\rho^{n-1}\left(\lambda_{k,n}^+\omega_{k,n}^+
	+\lambda_{k,n}^-\omega_{k,n}^-\right)\quad\text{in $T_k$},
\end{equation}
where $\rho\leq \frac{\varepsilon}{4}$ is a fixed positive number, and
$\bm\lambda=(\lambda_{k,n}^s)_{(k,n,s)\in\Lambda}\in\ellinf$ is a sequence of
complex numbers to be determined as a function of $(t,\x)$.  Observe that,
formally, $\wt{\omega}(\x,\bm\lambda)$ has the desired periods.  Regarding
convergence, by Lemma \ref{omega-lemma2}, we have in $\Omega_k$:
\begin{equation}
	\label{omega-estimate1}
	\sum_{n=2}^{\infty}\rho^{n-1}\left|(\lambda_{k,n}^+\omega_{k,n}^+
	+\lambda_{k,n}^-\omega_{k,n}^-)\right|
	\leq 2C\|\bm\lambda\|_{\infty}\sum_{n=2}^{\infty}\left(\frac{2\rho}{\varepsilon}\right)^{n-1}
	\leq 2C\|\bm\lambda\|_{\infty}.
\end{equation}
so the series \eqref{eq-omega-series} converges
absolutely in $\Omega_k$ and by Lemma \ref{omega-lemma1},
\begin{equation} \label{omega-estimate2}
	\|\wt{\omega}(\x,\bm\lambda)\|_{C^0(\Omega_k)}\leq C(1+\|\bm\lambda\|_{\infty}).
\end{equation}
We assume for now the convergence outside $\Omega_k$.

\begin{lemma} \label{omega-lemma3}
	Assume that the series \eqref{eq-omega-series} converges in the annulus
	$A_k^{\pm}$ for all $k\in\Z$. For $t\neq 0$, $\wt{\omega}(\x,\bm\lambda)$ is
	a well-defined 1-form on $\Sigma[t,\x]$ if and only if for all $k\in\Z$ and
	$n\geq 2$,
	\begin{align*}
		\lambda_{k,n}^+&=\frac{-1}{2\pi\ii}\int_{\partial D_{k+1}^-}\left(\frac{t^2}{\rho \,z_{k+1}^-}\right)^{n-1}\wt{\omega}(\x,\bm\lambda)
		\quad\text{and}\quad\\
		\lambda_{k,n}^-&=\frac{-1}{2\pi\ii}\int_{\partial D_{k-1}^+}\left(\frac{t^2}{\rho \,z_{k-1}^+}\right)^{n-1}\wt{\omega}(\x,\bm\lambda).
	\end{align*}
\end{lemma}

\begin{proof}
	Fix $t\neq 0$ and define a diffeomorphism
	\[
		\varphi_k=(z_{k+1}^-)^{-1}\circ\frac{t^2}{z_k^+}:A_k^+\to A_{k+1}^-
	\]
	so $z\in A_k^+$ is identified with $\varphi_k(z)\in A_{k+1}^-$ when opening nodes.
	Then $\wt{\omega}$ is well-defined on $\Sigma$ if and only if $\varphi_k^*\wt{\omega}=\wt{\omega}$ in the annulus $A_k^+$ for all
	$k\in\Z$. Using the theorem on Laurent series, this is equivalent to
	\begin{equation}
		\label{eq-laurent}
		\int_{\partial D_k^+}(z_k^+)^n(\varphi_k^*\wt{\omega}-\wt{\omega})=0\quad
		\text{for all $k\in\Z$ and $n\in\Z$.}
	\end{equation}
	We have
	\begin{align*}
		\int_{\partial D_k^+}(z_k^+)^n\varphi_k^*\wt{\omega}
		&=\int_{\partial D_k^+}\varphi_k^*\left[\left(\frac{t^2}{z_{k+1}^-}\right)^n\wt{\omega}\right]
		&\quad&\text{by definition of $\varphi_k$}\\
		&=\int_{\varphi_k(\partial D_k^+)}\left(\frac{t^2}{z_{k+1}^-}\right)^n\wt{\omega}
		&\quad&\text{by a change of variable}\\
		&=-\int_{\partial D_{k+1}^-}\left(\frac{t^2}{z_{k+1}^-}\right)^n\wt{\omega}
		&\quad&\text{because $\wt{\omega}$ is holomorphic in $A_{k+1}^-$.}
	\end{align*}
	For $n=0$, \eqref{eq-laurent} is always satisfied, because
	\[
		-\int_{\partial D_{k+1}^-}\wt{\omega}-\int_{\partial D_k^+}\wt{\omega}=-2\pi\ii\,\Res_{0_{k+1}}(\omega_{k+1,0_{k+1},v_{k+1}})-2\pi\ii\,\Res_{v_k}(\omega_{k,0_k,v_k})=0.
	\]
	For $n\geq 1$, we have by the Residue Theorem and definition of $\wt{\omega}(\x,\bm\lambda)$
	\[
		\int_{\partial D_k^+}(z_k^+)^n\wt{\omega}=2\pi\ii\,\rho^n\lambda_{k,n+1}^+,
	\]
	so \eqref{eq-laurent} is equivalent to
	\[
		2\pi\ii\,\rho^n\lambda_{k,n+1}^+=-\int_{\partial D_{k+1}^-}\left(\frac{t^2}{z_{k+1}^-}\right)^n\wt{\omega}.
	\]
	For $n\leq -1$, we have by the Residue Theorem 
	\[
		\int_{\partial D_{k+1}^-}\left(\frac{t^2}{z_{k+1}^-}\right)^n\wt{\omega}
		=2\pi\ii\, t^{2n} \rho^{-n}\lambda_{k+1,-n+1}^-
	\]
	so \eqref{eq-laurent} is equivalent to
	\[
		2\pi\ii \,t^{2n}\rho^{-n}\lambda_{k+1,-n+1}^-=-\int_{\partial D_k^+}(z_k^+)^n\wt{\omega}
	\]
	which, after replacing $n$ by $-n$ and $k$ by $k-1$, becomes
	\[
		2\pi\ii \,\rho^n\lambda_{k,n+1}^-=-\int_{\partial D_{k-1}^+}\left(\frac{t^2}{z_{k-1}^+}\right)^n\wt{\omega}
	\]
	for all $n\geq 1$. Collecting all results gives Lemma \ref{omega-lemma3}.
\end{proof}

\medskip

In view of Lemma \ref{omega-lemma3}, we define
\[
	L_{k,n}^{\pm}(t,\x,\bm\lambda)=\frac{-1}{2\pi\ii}\int_{\partial D_{k\pm
	1}^{\mp}}\left( \frac{t^2}{\rho\,z_{k\pm
	1}^{\mp}}\right)^{n-1}\wt{\omega}(\x,\bm\lambda)
\]
and $\bm L=(L_{k,n}^s)_{(k,n,s)\in\Lambda}$, so $\wt{\omega}(\x,\bm\lambda)$ is
well-defined on $\Sigma[t,\x]$ if and only if $\bm\lambda=\bm
L(t,\x,\bm\lambda)$.  Observe that $\bm L(t,\x,\bm\lambda)$ is defined for all
$\bm\lambda\in\ellinf$.  In particular, we do not need the convergence in
$A_k^{\pm}$ to define $\bm L$.  Also, $\wt\omega$ and $\bm L$ are affine with
respect to $\bm\lambda$.
\begin{lemma} \label{omega-lemma4}
	For $(t,\x)$ in a neighborhood of $(0,\cv{\x})$, $\bm\lambda\mapsto \bm
	L(t,\x,\bm\lambda)$ is contracting from $\ellinf$ to itself, hence has a
	fixed point $\bm\lambda(t,\x)$ by the Fixed Point Theorem.
\end{lemma}

\begin{proof}
	We can bound the length of $\partial D_{k}^{\pm}$ by a uniform constant
	$\ell$.  By Estimate \eqref{omega-estimate1}:
	\[
		|L_{k,n}^{\pm}(t,\x,\bm\lambda)-L_{k,n}^{\pm}(t,\x,\bm 0)|\leq
		\frac{\ell}{2\pi}\left(\frac{t^2}{\rho\varepsilon}\right)^{n-1}2C\|\bm\lambda\|_{\infty}.
	\]
	Hence if $t^2\leq\rho\varepsilon$,
	\[
		\|\bm L(t,\x,\bm\lambda)-\bm L(t,\x,\bm 0)\|_{\infty}\leq \frac{C\ell
		t^2}{\pi\rho\varepsilon}\|\bm\lambda\|_{\infty}.
	\]
	so $\bm L$ is contracting for $t$ sufficiently small.
\end{proof}

\medskip

Now we verify the convergence of~\eqref{eq-omega-series} outside $\Omega_k$.

\begin{lemma} \label{omega-lemma5}
	If $\bm\lambda=\bm L(t,\x,\bm\lambda)$ and $t \ne 0$ is sufficiently small,
	the series \eqref{eq-omega-series} converges absolutely in the annulus
	$A_k^{\pm}$ for all $k\in\Z$.
\end{lemma}

\begin{proof}
	We only deal with the convergence of $\sum_{n\geq 2} \rho^{n-1}
	\lambda_{k,n}^+ \omega_{k,n}^+$.  Its convergence in $A_k^-$ is
	straightforward because $\omega_{k,n}^+$ is holomorphic in $D_k^-$ and we
	already know the convergence on $\partial D_k^-$.  It remains to prove the
	convergence in $A_k^+$.

	By Cauchy Theorem, we can replace the circle $\partial D_{k+1}^-$ by the
	circle $|z_{k+1}^-|=2\varepsilon$ in the definition of
	$L_{k,n}^+(t,\x,\bm\lambda)$. Using Estimate \eqref{omega-estimate2}, this
	gives
	\begin{equation} \label{omega-estimate3}
		|\lambda_{k,n}^+|=|L_{k,n}^+(t,\x,\bm\lambda)|
		\leq C(1+\|\bm\lambda\|_{\infty})\left(\frac{t^2}{2\rho\varepsilon}\right)^{n-1}.
	\end{equation}
	By definition, the function
	\[
		\frac{\omega_{k,n}^+}{dz}-\frac{(z_k^+)'}{(z_k^+)^n}
	\]
	extends holomorphically to the disk $D_k^+$. By the maximum principle and
	Lemma \ref{omega-lemma2}
	\[
		\sup_{D_k^+}\left|\frac{\omega_{k,n}^+}{dz}-\frac{(z_k^+)'}{(z_k^+)^n}\right|
		=\max_{\partial D_k^+}\left|\frac{\omega_{k,n}^+}{dz}-\frac{(z_k^+)'}{(z_k^+)^n}\right|
		\leq C\left(\frac{2}{\varepsilon}\right)^{n-1}+\frac{C}{\varepsilon^n}\leq C\left(\frac{2}{\varepsilon}\right)^n.
	\]
	Hence recalling the definition of $A_k^+$, and provided $2t^2\leq
	\varepsilon^2$:
	\[
		\sup_{A_k^+}|\omega_{k,n}^+|\leq C\left(\frac{2}{\varepsilon}\right)^n+C\left(\frac{\varepsilon}{t^2}\right)^n\leq C\left(\frac{\varepsilon}{t^2}\right)^n.
	\]
	Using Estimate \eqref{omega-estimate3}, we obtain
	\[
		\sup_{A_k^+}\rho^{n-1}|\lambda_{k,n}^+\omega_{k,n}^+|\leq \frac{C\varepsilon}{t^2 2^n}(1+\|\bm\lambda\|_{\infty}).
	\]
	Hence the series $\sum_{n\geq 2}
	\rho^{n-1}\lambda_{k,n}^+\omega_{k,n}^+$ converges absolutely in
	$A_k^+$.

	Convergence of $\sum_{n\geq 2} \rho^{n-1} \lambda_{k,n}^- \omega_{k,n}^-$
	follows similarly.
\end{proof}

\medskip

We define $\omega[t,\x]=\wt{\omega}(t,\bm\lambda(t,\x))$.  By Lemmas
\ref{omega-lemma3} and \ref{omega-lemma5}, $\omega[t,\x]$ is a well-defined
holomorphic 1-form on $\Sigma[t,\x]$ and has the desired periods by definition.
This proves Proposition \ref{prop:omega}(a).  At $t=0$, $\bm L=\bm 0$ so
$\bm\lambda=\bm 0$ and $\omega=\omega_{k,0_k,v_k}$ in $T_k$.
Proposition~\ref{prop:omega}(b) follows from Lemma \ref{omega-lemma1}.

\subsection{Smooth dependence on parameters}

We denote $\wt{\gamma}_k^-$ the circle $|z|=2\varepsilon'$ in
$\T=\C/(\Z+\ii\Z)$ and $\wt{\gamma}_k^+$ the circle
$|z-\wt{v}_k|=2\varepsilon'$.  These two circles are fixed and included in the
domain $\wt{\Omega}_k$.  Define for $\eta\in C^0(\wt{\Omega})$
\[
	\wt{L}_{k,n}^{\pm}(t,\x,\eta)=\frac{-1}{2\pi\ii}\int_{\wt{\gamma}_{k\pm 1}^{\mp}}\left(
	\frac{t^2}{\rho\,z_{k\pm 1}^{\mp}\circ\psi_{k\pm 1}}\right)^{n-1}\eta
\]
and let $\wt{\bm L}=(\wt{L}_{k,n}^s)_{(k,n,s)\in\Lambda}.$
Using a change of variable
\[
	\wt{L}_{k,n}^{\pm}(t,\x,(\psi^*\wt{\omega})(\x,\bm\lambda))=\frac{-1}{2\pi\ii}
	\int_{\psi_{k\pm 1}(\wt{\gamma}_{\pm 1}^{\mp})}\left(\frac{t^2}{\rho\,z_{k\pm 1}^{\mp}}\right)^{n-1}\wt{\omega}(\x,\bm\lambda)
	=L_{k,n}^{\pm}(t,\x,\bm\lambda).
\]
Hence
\begin{equation} \label{eq-F-wtF}
	\bm L(t,\x,\bm\lambda)=\wt{\bm L}(t,\x,(\psi^*\wt{\omega})(\x,\bm\lambda)).
\end{equation}
By Lemma \ref{omega-lemma6} below and composition, $\bm L$ is smooth so its fixed point $\bm\lambda(t,\x)$ depends smoothly on $(t,\x)$. By the first point of Lemma \ref{omega-lemma6}, $(\psi^*\omega)[t,\x]=(\psi^*\wt{\omega})(\x,\bm\lambda(t,\x))$ depends smoothly on $(t,\x)$. This proves Proposition \ref{prop:omega}(c).
\begin{lemma} \label{omega-lemma6}
	\leavevmode
	\begin{enumerate}
		\item $\psi^*\wt{\omega}$ is a smooth function of $\x$ in an
			$\ellinf$-neighborhood of $\cv{\x}$ and
			$\bm\lambda\in\ellinf$, with value in $C^0(\wt{\Omega})$.

		\item $\wt{\bm L}$ is a smooth function of $t$ in a neighborhood of $0$,
			$\x$ in an $\ellinf$-neighborhood of $\cv{\x}$ and $\eta\in
			C^0(\wt{\Omega})$, with value in $\ellinf$.
	\end{enumerate}
\end{lemma}

\medskip

For any infinite set $K$, if $(V_k)_{k\in K}$ is a sequence of normed spaces,
we denote
\[
	\big(\bigoplus_{k\in K}V_k\big)_{\infty}=\{x\in\prod_{k\in
	K}V_k:\|x\|_{\infty}=\sup_{k\in K} \|x_k\|<\infty\}.
\]
To prove Lemma \ref{omega-lemma6}, we use the following elementary fact.
\begin{proposition} \label{prop:smooth1}
	For $k\in K$, let $f_k:B(0,r)\subset U_k\to V_k$ be a smooth function between  normed
	spaces. Assume that there exists uniform constants $C(m)$ such that
	\[
		\forall m\in\N,\quad\forall k\in K,\quad \forall x_k\in B(0,r),\quad \|d^m f_k(x_k)\|\leq C(m)
	\]
	where $d^m f_k$ denotes the $m$-th order differential of $f_k$.  Let
	$U^\infty=(\bigoplus_{k\in K} U_k)_{\infty}$ and $V^\infty=(\bigoplus_{k\in K}
	V_k)_{\infty}$.  Define $\bm f:B(0,r)\subset U^\infty\to V^\infty$ by $\bm
	f(\x)=(f_k(x_k))_{k\in K}$.  Then $\bm f$ is smooth and
	$d\bm f(\x)\bm h=(df_k(x_k)h_k)_{k\in K}$.
\end{proposition}

We summarize the hypothesis of Proposition \ref{prop:smooth1} by saying that
the functions $f_k$ have \emph{uniformly bounded derivatives}.

\begin{proof}
	It is straightforward to prove that $f$ is differentiable (with the indicated
	differential) using Taylor Formula with integral remainder.  Smoothness
	follows by induction.
\end{proof}

\medskip

Recall that $\Omega_{k,r}$ denotes the torus $T_k$ minus the disks
$|z_k^{\pm}|\leq r$.  We fix a uniform positive $\varepsilon''<\varepsilon$ so
that for all $\x$ in a neighborhood of $\cv{\x}$ and $k\in\Z$,
$\psi_k(\wt{\Omega}_k)\subset\Omega_{k,\varepsilon''}$.  We choose $\rho$ such
that $\rho\leq\varepsilon''/4$.
\begin{claim} \label{claim-omega}
	\leavevmode
	\begin{enumerate}
		\item For $k\in\Z$, $\psi_k^*\omega_{k,0_k,v_k}$ is a smooth function of
			$(\tau_k,v_k)$ in a neighborhood of $(\cv{\tau}_k,\cv{v}_k)$, with value
			in $C^0(\wt{\Omega}_k)$ and has uniformly bounded derivatives.

		\item For $k\in\Z$ and $n\geq 2$,
			$(\frac{\varepsilon''}{2})^{n-1}\psi_k^*\omega_{k,n}^{\pm}$ is a smooth
			function of $x_k=(a_k,b_k,v_k,\tau_k)$ in a neighborhood of $\cv{x}_k$,
			with value in $C^0(\wt{\Omega}_k)$, and has uniformly (with respect to $k$
			and $n$) bounded derivatives.

		\item If $t^2<\rho\varepsilon''$, then for $k\in\Z$ and $n\geq 2$,
			$\left(\frac{t^2}{\rho z_k^{\pm}\circ\psi_k}\right)^{n-1}$ restricted to
			$\wt{\gamma}_k^{\pm}$ is a smooth function of $t$ in a neighborhood of
			$0$ and $x_k$ in a neighborhood of $\cv{x}_k$, with value in
			$C^0(\wt{\gamma}_k^{\pm})$, and has uniformly bounded derivatives.
	\end{enumerate}
\end{claim}

\begin{proof}
	\leavevmode
	\begin{enumerate}
		\item follows from the explicit formula in Lemma \ref{omega-lemma1}
			and Hypothesis \ref{hyp:separate}.

		\item By \eqref{estimate-eta}, we have, for some uniform constant $C$
			\[
				\|\eta_{k,n}^{\pm}\|_{C^0(\Omega_{k,\varepsilon''})}\leq C\left(\frac{2}{\varepsilon''}\right)^{n-1}.
			\]
			Hence
			\[
				\|\psi_k^*\eta_{k,n}^{\pm}\|_{C^0(\wt{\Omega}_k)}\leq 2C\left(\frac{2}{\varepsilon''}\right)^{n-1}.
			\]
			Observe that $\psi_k$ depends holomorphically on $\tau_k$.  Also,
			$\eta_{k,n}^{\pm}$ depends holomorphically on $x_k$ ($\omega_{k,n}^{\pm}$
			does not).  Hence for $z\in\wt{\Omega}_k$, $\psi_k^*\eta_{k,n}^{\pm}(z)$
			depends holomorphically on $x_k$. By Cauchy Estimate, restricting the
			parameter $\x$ to a smaller neighborhood of $\cv{\x}$,we have
			\[
				\|d^m\psi_k^*\eta_{k,n}^{\pm}\|_{C^0(\wt{\Omega}_k)}\leq C(m)\left(\frac{2}{\varepsilon''}\right)^{n-1}
			\]
			for some uniform constants $C(m)$, where $d^m$ denotes the $m$-th order
			differential with respect to $x_k$. (b) follows from \eqref{omega-eta}.

		\item Since
			$\psi_k(\wt{\gamma}_k^{\pm})\subset\psi_k(\wt{\Omega}_k)\subset\Omega_{k,\varepsilon''}$,
			we have $|z_k^{\pm}\circ\psi_k|\geq \varepsilon''$ on
			$\wt{\gamma}_k^{\pm}$.  Hence
			\[
				\left\|\left(\frac{t^2}{\rho z_k^{\pm}\circ\psi_k}\right)^{n-1}\right\|_{C^0(\wt{\gamma}_k^{\pm})}\leq 1.
			\]
			Since $z_k^{\pm}\circ\psi_k(z)$ depends holomorphically on $x_k$, we
			obtain uniform estimates of the derivatives by Cauchy Estimate.
	\end{enumerate}
\end{proof}

\medskip

\begin{proof}[Proof of Lemma \ref{omega-lemma6}]
	We write $\wt{\omega}_1$ for the first term in the definition of
	$\wt{\omega}$ and $\wt{\omega}_2$ for the second term (the sum for $n\geq
	2$).
	\begin{enumerate}
		\item $(\psi^*\wt{\omega}_1)(\x)\in C^0(\wt{\Omega})$ and is a smooth
			function of $\x$. This follows from Claim \ref{claim-omega}(a) and
			Proposition \ref{prop:smooth1}.  $(\psi^*\wt{\omega}_2)(\x,\bm\lambda)\in
			C^0(\wt{\Omega})$ and is a smooth function of $\x$ and $\bm\lambda$.
			This follows from Claim \ref{claim-omega}(b), Proposition
			\ref{prop:smooth1} and the fact that the bilinear operator
			\begin{equation} \label{eq-bilinear-op1}
				\big(\eta,\bm\lambda)\mapsto \left(\sum_{n=2}^{\infty}
					\left(\frac{2\rho}{\varepsilon''}\right)^{n-1}
					(\lambda_{k,n}^+\eta_{k,n}^+ 
					+\lambda_{k,n}^-\eta_{k,n}^-)
				\right)_{k\in\Z}
			\end{equation}
			is bounded from $\big(\bigoplus_{(k,n,s)\in\Lambda}
			C^0(\wt{\Omega}_k)\big)_{\infty} \times\ellinf$ to $C^0(\wt{\Omega})$.
			(Since $\rho\leq\varepsilon''/4$, it has norm at most 2).  This proves
			Lemma \ref{omega-lemma6}(a).

		\item Lemma \ref{omega-lemma6}(b) follows from Claim \ref{claim-omega}(c),
			Proposition \ref{prop:smooth1} and the fact that the bilinear operator
			\begin{equation}
				\label{eq-bilinear2}
				(\bm f,\eta)\mapsto\left(\frac{-1}{2\pi\ii}\int_{\wt{\gamma}_{k+s}^{-s}}f_{k+s,n}^{-s}\eta\right)_{(k,n,s)\in\Lambda}
			\end{equation}
			is bounded from $\big(\bigoplus_{(k,n,s)\in\Lambda}C^0(\wt{\gamma}_k^s)\big)_{\infty}\times C^0(\wt{\Omega})$ to $\ellinf$.
	\end{enumerate}
\end{proof}

\section{Asymptotic behavior}\label{sec:asymptotic}

Assume that we are given two configurations $(q_k)$ and $(q_k')$ such that
$q_k=q'_k$ for all $k\geq 0$.  In this section we prove that the holomorphic
1-form $\omega'$ and the parameters $\x'$ are asymptotic to $\omega$ and $\x$.
These results were used in Section~\ref{sec:asymptotic-immersion} to prove the
asymptotic behaviors of the minimal surfaces.

\subsection{Asymptotic behavior of \texorpdfstring{$\omega$}{[omega]}}\label{ssec:asymptotic-omega}

For any infinite set $K$ equipped with a weight function
$\sigma:K\to[1,\infty)$, if $(V_k)_{k\in K}$ is a sequence of normed spaces, we
define
\[
	\big(\bigoplus_{k\in K}V_k\big)_{\infty,\sigma}=\{\x\in\prod_{k\in K}V_k:\|\x\|_{\infty,\sigma}=\sup_{k\in K} \sigma(k)\|x_k\|<\infty\}.
\]
If $V_k = \C^n$ for all $k \in K$, then we simply use the notation
$\ell^{\infty,\sigma}(K)$.  The argument $K$ will be omitted if it is clear in
the context.

\medskip

Let $\cv{\x}$ and $\cv{\x}'$ be the central values corresponding to the
configurations $(q_k)$ and $(q'_k)$, as given by \eqref{eq:central-value}.  Our
goal is to compare $\omega[t,\x]$ to $\omega[t,\x']$ for $\x,\x'$ in a
neighborhood of $\cv{\x},\cv{\x}'$.  For this purpose, we replace the
definition of $\wt\Omega_k$ and $\wt{\Omega}$ in Section \ref{ssec:dh} by
\[
	\wt{\Omega}_k=\T\setminus\big(\overline{D}(0,\varepsilon')\cup \overline{D}(\cv{\wt{v}}_k,\varepsilon')\cup \overline{D}(\cv{\wt{v}}'_k,\varepsilon')\big)
	\quad\text{and} \quad
	\wt{\Omega}=\bigsqcup_{k\in\Z}\wt{\Omega}_k.
\]
Now $(\psi^*\omega)[t,\x]$ and $(\psi^*\omega)[t,\x']$ are both defined on the
same fixed domain $\wt{\Omega}$ so we can compare them.

\medskip

Fix $\delta>1$ and define the weight $\sigma:\Z\to[1,\infty)$ by $\sigma(k)=1$
if $k\leq 0$ and $\sigma(k)=\delta^k$ if $k\geq 0$.  We extend this weight to
$\Lambda$ by $\sigma(k,n,s)=\sigma(k)$ for all $(k,n,s)\in\Lambda$.  We have
$\cv{x}_k=\cv{x}_k'$ for all $k\geq 0$ so
$\cv{\x}'-\cv{\x}\in\ell^{\infty,\sigma}$.  It will be convenient to
write 
$$\Delta\x=\x'-\x.$$
We define the weighted space
$C^{0,\sigma}(\wt{\Omega})$ (not to be confused with a H\"older space) by
\[
	C^{0,\sigma}(\wt{\Omega})=\big(\bigoplus_{k\in\Z}C^0(\wt{\Omega}_k)\big)_{\infty,\sigma}
\]
So functions in $C^{0,\sigma}(\wt{\Omega})$ decay like $\delta^{-k}$ in
$\wt{\Omega}_k$ as $k\to+\infty$.

\begin{proposition}
	\label{prop:omega-decay}
	For $t$ small enough, $\x$ in an $\ell^{\infty}$-neighborhood of $\cv{\x}$ and
	$\Delta\x$ in an $\ell^{\infty,\sigma}$ neighborhood of $\Delta\cv\x=\cv\x'-\cv\x$,
	\[
		(\psi^*\omega)[t,\x+\Delta\x]-(\psi^*\omega)[t,\x] \in
		C^{0,\sigma}(\wt{\Omega})
	\]
	and depends smoothly on $t$, $\x\in\ell^{\infty}$ and
	$\Delta\x\in\ell^{\infty,\sigma}$.
\end{proposition}

\begin{proof}
	Recall that $\bm\lambda(t,\x)$ is given by Lemma \ref{omega-lemma4} and
	define
	\[
		\bm\mu(t,\x,\Delta\x)=\bm\lambda(t,\x+\Delta\x)-\bm\lambda(t,\x)\in\ell^{\infty}(\Lambda).
	\]
	Our goal is to prove that $\bm\mu(t,\x,\Delta\x)\in\ell^{\infty,\sigma}(\Lambda)$.

	Let
	\[
		\Delta\bm L(t,\x,\Delta\x,\bm\lambda,\Delta\bm\lambda)=\bm L(t,\x+\Delta\x,\bm\lambda+\Delta\bm\lambda)-\bm L(t,\x,\bm\lambda).
	\]
	Then
	\[
		\Delta\bm L(t,\x,\Delta\x,\bm\lambda(t,\x),\bm\mu(t,\x,\Delta\x))
		=\bm L(t,\x+\Delta\x,\bm\lambda(t,\x+\Delta\x))-\bm L(t,\x,\bm\lambda(t,\x))\\
		=\bm\mu(t,\x,\Delta\x).
	\]
	In other words, $\bm\mu(t,\x,\Delta\x)$ is a fixed point of $\Delta\bm L$ with
	respect to the $\Delta{\bm\lambda}$ variable.

	Define
	\[
		\Delta\wt{\bm L}(t,\x,\Delta\x,\eta,\Delta\eta)=\wt{\bm
		L}(t,\x+\Delta\x,\eta+\Delta\eta)-\wt{\bm L}(t,\x,\eta)
	\]
	and
	\[
		\bm H(\x,\Delta{\x},\bm\lambda,\Delta{\bm\lambda})=(\psi^*\wt{\omega})(\x+\Delta{\x},\bm\lambda+\Delta{\bm\lambda})-(\psi^*\wt{\omega})(\x,\bm\lambda).
	\]
	Recalling \eqref{eq-F-wtF}, we have
	\[
		\Delta\bm L(t,\x,\Delta{\x},\bm\lambda,\Delta{\bm\lambda})=\Delta \wt{\bm L}\big(t,\x,\Delta{\x},(\psi^*\wt{\omega})(\x,\bm\lambda),
		\bm H(\x,\Delta{\x},\bm\lambda,\Delta{\bm\lambda})\big).
	\]
	By Lemma \ref{decay-lemma1} below, $\Delta\bm L \in \ell^{\infty,\sigma}$ is
	smooth with respect to $\x, \bm\lambda \in\ellinf$ and $\Delta{\x},
	\Delta\bm\lambda \in \ell^{\infty,\sigma}$, and is contracting with respect
	to $\Delta{\bm\lambda}$, so its unique fixed point $\bm\mu(t,\x,\Delta{\x})$
	is in $\ell^{\infty,\sigma}$ and depends smoothly on $t,\x$,$\Delta{\x}$ in
	their respective spaces.  Finally, we have by definition
	\begin{align*}
 		(\psi^*\omega)[t,\x+\Delta{\x}]-(\psi^*\omega)[t,\x]
 		&=(\psi^*\wt{\omega})(\x+\Delta{\x},\bm\lambda(t,\x+\Delta{\x}))-(\psi^*\wt{\omega})(\x,\bm\lambda(t,\x))\\
		&=\bm H\big(\x,\Delta{\x},\bm\lambda(t,\x),\bm \mu(t,\x,\Delta{\x})\big)
	\end{align*}
	so Proposition \ref{prop:omega-decay} follows from Lemma \ref{decay-lemma1}(a).
\end{proof}

\begin{lemma} \label{decay-lemma1}
	\leavevmode
	\begin{enumerate}
		\item For $\x$ in an $\ell^{\infty}$-neighborhood of $\cv{\x}$, $\Delta{\x}$ in an $\ell^{\infty,\sigma}$-neighborhood of $\Delta\cv{\x}$, $\bm\lambda\in\ellinf$ and $\Delta{\bm\lambda}\in\ell^{\infty,\sigma}$,
			\[
				\bm H(\x,\Delta{\x},\bm\lambda,\Delta{\bm\lambda})\in C^{0,\sigma}(\wt{\Omega})
			\]
			and depends smoothly on $\x,\Delta{\x},\bm\lambda,\Delta{\bm\lambda}$.

		\item For $t$ in a neighborhood of $0$, $\x$ in an
			$\ell^{\infty}$-neighborhood of $\cv{\x}$, $\Delta{\x}$ in an
			$\ell^{\infty,\sigma}$-neighborhood of $\Delta\cv{\x}$, $\eta\in
			C^0(\wt{\Omega})$ and $\Delta{\eta}\in C^{0,\sigma}(\wt{\Omega})$,
			\[
				\Delta\wt{\bm L}(t,\x,\Delta{\x},\eta,\Delta{\eta})\in \ell^{\infty,\sigma}
			\]
			and depends smoothly on $t,\x,\Delta{\x},\eta,\Delta{\eta}$.

		\item For $t$ small enough, $\Delta\bm L$ is contracting with respect to
			$\Delta{\bm\lambda}$, as a map from $\ell^{\infty,\sigma}$ to itself.
	\end{enumerate}
\end{lemma}

We need the following

\begin{proposition}
	\label{prop:smooth2} Under the same notations and hypothesis as in
	Proposition \ref{prop:smooth1}, let $\sigma:K\to[1,\infty)$ be an arbitrary
	weight.  Let $U^{\infty,\sigma}=(\bigoplus_{k\in K}U_k)_{\infty,\sigma}$ and
	$V^{\infty,\sigma}=(\bigoplus_{k\in K}V_k)_{\infty,\sigma}$.  Define for
	$\x\in B(\bm 0,r/2)\subset U^\infty$ and $\Delta{\x}\in B(\bm 0,r/2)\subset
	U^{\infty,\sigma}$
	\[
		\Delta\bm f(\x,\Delta\x)=\big(f_k(x_k+\Delta x_k)-f_k(x_k)\big)_{k\in K}.
	\]
	Then $\Delta\bm f(\x,\Delta{\x})\in V^{\infty,\sigma}$, $\Delta \bm f$ is
	smooth and
	\[
		d(\Delta\bm f)(\x,\Delta\x)(\bm h,\Delta{\bm h})=\big(df_k(x_k+\Delta{x}_k)(h_k+\Delta{h}_k)-df_k(x_k)h_k\big)_{k\in K}.
	\]
\end{proposition}

\begin{proof}
	By the Mean Value Inequality
	\[
		\sigma(k)\|f_k(x_k+\Delta{x}_k)-f_k(x_k)\|\leq C\sigma(k)\|\Delta{x}_k\|
	\]
	Hence $\Delta\bm f(\x,\Delta{\x})\in V^{\infty,\sigma}$.  Define
	\[
		l_k(h_k,\Delta{h}_k)=df_k(x_k+\Delta{x}_k)(h_k+\Delta{h}_k)-df_k(x_k)h_k
		\quad\text{and}\quad \bm l(\bm h,\Delta{\bm h})=(l_k(h_k,\Delta{h}_k))_{k\in K}.
	\]
	Using the Mean Value Inequality, one easily obtains
	\[
		\sigma(k)\|l_k(h_k,\Delta{h}_k)\|\leq C\sigma(k)\big(\|\Delta{h}_k\|+\|\Delta{x}_k\|\,\|h_k\|\big).
	\]
	Hence $\bm l$ is a bounded operator from $U^\infty \times U^{\infty,\sigma}$ to
	$V^{\infty,\sigma}$.  Using Taylor Formula with integral remainder, we have
	\begin{align*}
		&\Delta f_k(x_k+h_k,\Delta{x}_k+\Delta{h}_k)-\Delta f_k(x_k,\Delta{x}_k)-l_k(h_k,\Delta{h}_k)\\
		=&f_k(x_k+\Delta{x}_k+h_k+\Delta{h}_k)-f_k(x_k+\Delta{x}_k)\\
		&-df_k(x_k+\Delta{x}_k)(h_k+\Delta{h}_k)
		-\left[f_k(x_k+h_k)-f_k(x_k)-df_k(x_k)h_k\right]\\
		=&\int_0^1(1-t)\left[d^2f_k\big(x_k+\Delta{x}_k+t(h_k+\Delta{h}_k)\big)(h_k+\Delta{h}_k)^2
		-d^2f_k(x_k+th_k)h_k^2\right]dt
	\end{align*}
	By the Mean Value Inequality
	\begin{align*}
		&\sigma(k)\|\Delta f_k(x_k+h_k,\Delta{x}_k+\Delta{h}_k)-\Delta f_k(x_k,\Delta{x}_k)-l_k(h_k,\Delta{h}_k)\|\\
		\leq&
		C\sigma(k)\left[ 2\|h_k\|\|\Delta{h}_k\|+\|\Delta{h}_k\|^2
		+(\|\Delta{x}_k\|+\|\Delta{h}_k\|)\|h_k\|^2 \right].
	\end{align*}
	Hence
	\[
		\|\Delta\bm f(\x+\bm h,\Delta{\x}+\Delta{\bm h})-\Delta\bm f(\x,\Delta{\bm x})-\bm l(\bm h,\Delta{\bm h})\|_{\infty,\sigma}=O((\|\bm h\|_{\infty}+\|\Delta{\bm h}\|_{\infty,\sigma})^2)
	\]
	so $\Delta\bm f$ is differentiable with $d(\Delta \bm f)(\x,\Delta{\x})=\bm
	l$. Smoothness follows by induction.
\end{proof}

\medskip

We shall use the following corollary with $K=\Z$, $K^+=\N$ and $K^-=\Z\setminus\N$:

\begin{corollary} \label{corollary-smooth2}
	With the same notation as in Proposition \ref{prop:smooth2}, assume that for
	$k\in K$, $f_k$ is defined in $B(\cv{x}_k,r)\cup B(\cv{x}'_k,r)$ in $U_k$ and
	has uniformly bounded derivatives. Assume that $K$ admits a partition
	$(K^+,K^-)$ such that for all $k\in K^+$, $\cv{x}_k=\cv{x}'_k$, and for all
	$k\in K^-$, $\sigma(k)=1$.  Then for $\x\in B(\cv{\x},r/2)\subset U^\infty$ and
	$\Delta{\x}\in B(\cv{\x}'-\cv{\x},r/2)\subset U^{\infty,\sigma}$, $\Delta\bm
	f(\x,\Delta{\x})\in V^{\infty,\sigma}$ and depends smoothly on $\x$ and $\Delta{\x}$.
\end{corollary}
\begin{proof}
	We decompose a sequence $\x=(x_k)_{k\in K}$ as $\x=\x^+ + \x^-$ with $\x^+$
	supported on $K^+$ and $\x^-$ supported on $K^-$.  By Proposition
	\ref{prop:smooth1}, $\bm f^-(\x^-)\in V^\infty$ and $\bm
	f^-(\x^-+\Delta{\x}^-)\in V^\infty$ so $\Delta \bm f^-(\x^-,\Delta{\x}^-)\in
	V^\infty$. Since $\sigma=1$ on $K^-$, $\Delta \bm f^-(\x^-,\Delta{\x}^-)\in
	V^{\infty,\sigma}$.  Since $\cv{\x}^+=(\cv{\x}')^+$, we may use the change of
	variable $\x^+=\cv{\x}^+ + \bm y^+$ and conclude that $\Delta \bm
	f^+(\x^+,\Delta{\x}^+)\in V^{\infty,\sigma}$ by Proposition
	\ref{prop:smooth2}.
\end{proof}

\medskip

\begin{proof}[Proof of Lemma \ref{decay-lemma1}]
	Lemma \ref{decay-lemma1}(a) follows from the following.
	\begin{itemize}
		\item $(\psi^*\wt{\omega}_1)(\x+\Delta{\x})-(\psi^*\wt{\omega}_1)(\x)\in
			C^{0,\sigma}(\wt{\Omega})$ and is a smooth function of $\x$,
			$\Delta{\x}$. This follows from Claim \ref{claim-omega}(a) and Corollary
			\ref{corollary-smooth2}.

		\item
			$(\psi^*\wt{\omega}_2)(\x+\Delta{\x},\bm\lambda)-(\psi^*\wt{\omega}_2)(\x,\bm\lambda)\in
			C^{0,\sigma}(\wt{\Omega})$ and is a smooth function of $\x$, $\Delta{\x}$
			and $\bm\lambda$.  This follows from Claim \ref{claim-omega}(b),
			Corollary \ref{corollary-smooth2} and the fact that the bilinear operator
			\eqref{eq-bilinear-op1} is bounded from
			$\big(\bigoplus_{(k,n,s)\in\Lambda}
			C^0(\wt{\Omega}_k)\big)_{\infty,\sigma} \times\ell^{\infty}(\Lambda)$ to
			$C^{0,\sigma}(\wt{\Omega})$.

		\item $(\psi^*\wt{\omega}_2)(\x+\Delta{\x},\Delta{\bm\lambda})\in
			C^{0,\sigma}(\wt{\Omega})$ and is a smooth function of $\x$, $\Delta{\x}$
			and $\Delta{\bm\lambda}$. This follows from Claim \ref{claim-omega}(b),
			Proposition \ref{prop:smooth1} and the fact that the bilinear operator
			\eqref{eq-bilinear-op1} is bounded from
			$\big(\bigoplus_{(k,n,s)\in\Lambda} C^0(\wt{\Omega}_k)\big)_{\infty}
			\times\ell^{\infty,\sigma}(\Lambda)$ to $C^{0,\sigma}(\wt{\Omega})$.
	\end{itemize}

	Lemma \ref{decay-lemma1}(b) follows from the following.

	\begin{itemize}
		\item $\wt{\bm L}(t,\x+\Delta{\x},\eta)-\wt{\bm
			L}(t,\x,\eta)\in\ell^{\infty,\sigma}(\Lambda)$ and depends smoothly on
			$t,\x,\Delta{\x}$ and $\eta$. This follows from Claim
			\ref{claim-omega}(c), Corollary \ref{corollary-smooth2} and the fact that
			the bilinear operator \eqref{eq-bilinear2} is bounded from
			$\big(\bigoplus_{(k,n,s)\in\Lambda}C^0(\wt{\gamma}_k^s)\big)_{\infty,\sigma}\times
			C^0(\wt{\Omega})$ to $\ell^{\infty,\sigma}(\Lambda)$.  This uses that
			$\frac{\sigma(k)}{\sigma(k\pm 1)}\leq\delta$ and explains our choice of
			the weight $\sigma$.

		\item $\wt{\bm
			L}(t,\x+\Delta{\x},\Delta{\eta})\in\ell^{\infty,\sigma}(\Lambda)$ and
			depends smoothly on $t,\x,\Delta{\x}$ and $\Delta{\eta}$. This follows
			from Claim \ref{claim-omega}(c), Proposition \ref{prop:smooth1} and the
			fact that the bilinear operator \eqref{eq-bilinear2} is bounded from
			$\big(\bigoplus_{(k,n,s)\in\Lambda}C^0(\wt{\gamma}_k^s)\big)_{\infty}\times
			C^{0,\sigma}(\wt{\Omega})$ to $\ell^{\infty,\sigma}(\Lambda)$.
	\end{itemize}

	Finally, we have
	\[
		\Delta \bm L(t,\x,\Delta{\x},\bm\lambda,\Delta{\bm\lambda})-\Delta \bm L(t,\x,\Delta{\x},\bm\lambda,0)=
		\bm L(t,\x+\Delta{\x},\bm\lambda+\Delta{\bm\lambda})-\bm L(t,\x+\Delta{\x},\bm\lambda).
	\]
	Using Estimate \eqref{omega-estimate1} as in the proof of Lemma
	\ref{omega-lemma4},
	\[
		\|\bm L(t,\x+\Delta{\x},\bm\lambda+\Delta{\bm\lambda})-\bm L(t,\x+\Delta{\x},\bm\lambda)\|_{\infty,\sigma}
		\leq \frac{C\ell \delta t^2}{\pi\rho\varepsilon}\|\Delta{\bm\lambda}\|_{\infty,\sigma}
	\]
	so $\Delta \bm L$ is contracting with respect to $\Delta{\bm\lambda}$ for $t$
	small enough.
\end{proof}

\subsection{Asymptotic behavior of the parameters}

Let $\x(t)$ and $\x'(t)$ be the solutions obtained in Section
\ref{sec:proof} from the configurations $(q_k)$ and $(q'_k)$, respectively.

\begin{proposition} \label{prop:parameter-decay}
	For $t$ small enough, $\x'(t)-\x(t)\in\ell^{\infty,\sigma}$.
\end{proposition}

\begin{proof}
	Recall the definition of $\cal{E}_k$ in Section \ref{ssec:zeros},
	$\cal{P}_{k,1}$ and $\cal{P}_{k,2}$ in Section \ref{ssec:period-problem} and
	$\cal{G}_k$ in Section \ref{ssec:balancing}.  We have solved equations by
	three consecutive applications of the Implicit Function Theorem.  But we
	could have solved all of them by one single application.  Indeed, consider
	the change of parameter
	\[
		b_k=-a_k\xi(v_k,\tau_k)+\wh{b}_k.
	\]
	By the computations in Sections \ref{ssec:zeros},
	\ref{ssec:period-problem} and \ref{ssec:balancing}, the jacobian of
	$(\cal{E}_k,\cal{P}_{k,1},\cal{P}_{k,2},\cal{G}_k)$ with respect to
	$(\wh{b}_k,a_k,\tau_k,v_k)$ has upper-triangular form with $\R$-linear
	automorphisms of $\C$ on the diagonal, whose inverses are uniformly bounded
	with respect to $k$.  Define
	\[
		\cal{F}_{k}(t,\x)=\left(\cal{E}_k(t,\x),\cal{P}_{k,1}(t,\x),\cal{P}_{k,2}(t,\x),\cal{G}_k(t,\x)+2\pi\ii G(q_0;T)\right)
		\quad\text{and}\quad
		\bm{\cal F}=(\cal{F}_k)_{k\in\Z}.
	\]
	Then $d_{\x}\bm{\cal F}(0,\cv{\x}')$ is an automorphism of $\ell^{\infty}$,
	and restricts to an automorphism of $\ell^{\infty,\sigma}$.  Define, for $\x$
	in an $\ell^{\infty}$-neighborhood of $\cv{\x}$ and $\Delta{\x}$ in an
	$\ell^{\infty,\sigma}$-neighborhood of $\Delta\cv{\x}=\cv{\x}'-\cv{\x}$
	\[
		\Delta\bm{\cal F}(t,\x,\Delta{\x})=\bm{\cal F}(t,\x+\Delta{\x})-\bm{\cal F}(t,\x).
	\]
	By Lemma \ref{decay-lemma2} below, $\Delta\bm{\cal F}(t,\x,\Delta{\x})\in\ell^{\infty,\sigma}$.
	We have
	\[
		\Delta\bm{\cal F}(0,\cv{\x},\Delta\cv{\x})=\bm{\cal F}(t,\cv{\x}')-\bm{\cal F}(t,\cv{\x})=0
		\quad\text{and}\quad
		d_{\Delta{\x}}(\Delta\bm{\cal F})(0,\cv{\x},\Delta\cv{\x})=d_{\x}\bm{\cal F}(0,\cv{\x}').
	\]
	By the Implicit Function Theorem, for $t$ small enough and $\x$ in a
	neighborhood of $\cv{\x}$, there exists
	$\Delta{\x}(t,\x)\in\ell^{\infty,\sigma}$ such that $\Delta\bm{\cal
	F}(t,\x,\Delta{\x}(t,\x))=0$.  We substitute $\x=\x(t)$ and obtain $\bm{\cal
	F}(t,\x(t)+\Delta{\x}(t,\x(t)))=\bm{\cal F}(t,\x'(t))=0$.  By uniqueness,
	$\x'(t)=\x(t)+\Delta{\x}(t,\x(t))$, which proves Proposition
	\ref{prop:parameter-decay}.
\end{proof}

\begin{lemma} \label{decay-lemma2}
	For $t$ in a neighborhood of $0$, $\x$ in an $\ell^{\infty}$-neighborhood of
	$\cv{\x}$ and $\Delta{\x}$ in an $\ell^{\infty,\sigma}$-neighborhood of
	$\Delta\cv{\x}$, $\Delta\bm{\cal F}(t,\x,\Delta{\x})\in\ell^{\infty,\sigma}$
	and depends smoothly on $t$, $\x$ and $\Delta{\x}$.
\end{lemma}

\begin{proof}
	Define for $z\in \T$:
	\[
		f_k[\x](z)=\frac{\psi_k(z)-(Z_{k,1}+Z_{k,2})}{\psi_k(z)^2-(Z_{k,1}+Z_{k,2})\psi_k(z)+Z_{k,1}Z_{k,2}}
		\quad\text{and}\quad \bm f[\x]=(f_k[\x])_{k\in\Z}.
	\]
	By Cauchy Theorem and a change of variable,
	\[
		\cal{E}_k(t,\x)=\frac{1}{2\pi\ii}\int_{\partial\wt{\Omega}_{k,2\varepsilon'}}f_k\psi_k^*\omega[t,\x].
	\]
	Hence we can write
	\[
		\bm{\cal E}(t,\x)=\bm B(\bm f[\x],(\psi^*\omega)[t,\x])\quad\text{where}\quad
		\bm B(\bm f,\eta)=\left(\frac{1}{2\pi\ii}\int_{\partial\wt{\Omega}_{k,2\varepsilon'}}f_k \eta_k\right)_{k\in\Z}.
	\]
	Using Weierstrass Preparation Theorem, the symmetric functions of $Z_{k,1}$
	and $Z_{k,2}$ are holomorphic functions of $x_k$. Hence $f_k$ is a smooth
	function of $x_k$ with value in
	$C^0(\partial\wt{\Omega}_{k,2\varepsilon'})$.  Using
	Corollary~\ref{corollary-smooth2} and that the bilinear operator
	\[
		\bm B:\big(\bigoplus_{k\in\Z}C^0(\partial\wt{\Omega}_{k,2\varepsilon'})\big)_{\infty,\sigma}\times C^0(\wt{\Omega})\to\ell^{\infty,\sigma}
	\]
	is bounded, we conclude that
	\begin{equation}
		\label{eq-decay1}
		\bm B(\bm f[\x+\Delta{\x}]-\bm f[\x],(\psi^*\omega)[t,\x])\in\ell^{\infty,\sigma}
	\end{equation}
	and depends smoothly on $t$, $\x$ and $\Delta{\x}$.
	Using Proposition \ref{prop:omega-decay} and that the bilinear operator
	\[
		\bm B:\big(\bigoplus_{k\in\Z}C^0(\partial\wt{\Omega}_{k,2\varepsilon'})\big)_{\infty}\times C^{0,\sigma}(\wt{\Omega})\to\ell^{\infty,\sigma}
	\]
	is bounded, we obtain
	\begin{equation} \label{eq-decay2}
		\bm B(\bm f[\x+\Delta{\x}],(\psi^*\omega)[t,\x+\Delta{\x}]-(\psi^*\omega)[t,\x])\in\ell^{\infty,\sigma}
	\end{equation}
	and depends smoothly on $t$, $\x$ and $\Delta{\x}$.
	Adding \eqref{eq-decay1} and \eqref{eq-decay2}, we conclude that
	\[
		\bm{\cal E}(t,\x+\Delta{\x})-\bm{\cal E}(t,\x)\in\ell^{\infty,\sigma}
	\]
	and depends smoothly on $t$, $\x$ and $\Delta{\x}$.

	Finally, $\cal{P}_{k,1}$, $\cal{P}_{k,2}$ and $\cal{G}_k$ are defined as
	integrals of $\omega$ times powers of $g_k$ on certain curves in $T_k$, so we
	can deal with them in the same way as $\cal{E}_k$.
\end{proof}

\bibliography{References}
\bibliographystyle{alpha}

\end{document}